\documentclass[twoside,11pt]{article}

\usepackage{amsmath, mathtools}
\usepackage{smile}
\usepackage[linesnumbered,ruled,vlined]{algorithm2e}
\usepackage{tikz}
\usepackage{bbm}
\usetikzlibrary{shapes}
\usetikzlibrary{plotmarks}
\usepackage{xcolor}
\usepackage{capt-of}
\usepackage{fullpage}
\usepackage[colorlinks=true,
            linkcolor=blue,
            urlcolor=blue,
            citecolor=blue]{hyperref}

\usepackage[colorinlistoftodos]{todonotes}

\newcommand{\diam}{\operatorname{diam}}

\SetKwInput{KwInput}{Input}                
\SetKwInput{KwOutput}{Output}              
\SetKwRepeat{KwRepeat}{repeat}{until}

\def\T{{ ^\mathrm{\scriptscriptstyle T} }}

\begin{document}

\begin{center}
{\LARGE On the minimax rate of the Gaussian sequence model under bounded convex constraints}

{\large
\begin{center}
Matey Neykov
\end{center}}

{Department of Statistics \& Data Science\\ Carnegie Mellon University\\Pittsburgh, PA 15213\\[2ex]\texttt{mneykov@stat.cmu.edu}}
\end{center}

\begin{abstract} 
We determine the exact minimax rate of a Gaussian sequence model under bounded convex constraints, purely in terms of the local geometry of the given constraint set $K$. Our main result shows that the minimax risk (up to constant factors) under the squared $\ell_2$ loss is given by $\varepsilon^{*2} \wedge \diam(K)^2$ with
\begin{align*}
    \varepsilon^* = \sup \bigg\{\varepsilon : \frac{\varepsilon^2}{\sigma^2} \leq \log M^{\operatorname{loc}}(\varepsilon)\bigg\},
\end{align*}
where $\log M^{\operatorname{loc}}(\varepsilon)$ denotes the local entropy of the set $K$, and $\sigma^2$ is the variance of the noise. We utilize our abstract result to re-derive known minimax rates for some special sets $K$ such as hyperrectangles, ellipses, and more generally quadratically convex orthosymmetric sets. Finally, we extend our results to the unbounded case with known $\sigma^2$ to show that the minimax rate in that case is $\varepsilon^{*2}$.
\end{abstract}

\section{Introduction}

This paper focuses on the Gaussian sequence model $Y_i = \mu_i + \xi_i$ with $n$ observations (i.e., $i \in \{1,\ldots,n\}$), where $\xi_i \sim N(0,\sigma^2)$ are independent and identically distributed (i.i.d.), and the vector $\mu \in \mathbb{R}^n$ belongs to a known bounded convex set $K$. In particular we would like to determine the minimax rate for this problem. In detail, we would like to quantify (up to proportionality constants) the rate of the following expression, also known as the minimax risk:
\begin{align}\label{minimax:risk}
    \inf_{\hat \nu} \sup_{\mu \in K} \mathbb{E} \|\hat \nu(Y) - \mu\|^2,
\end{align}
where the infimum is taken with respect to all measurable functions (estimators) of the data, and we use the shorthand $\|\cdot\|$ for the Euclidean norm. The minimax risk may appear to be overly pessimistic to some, but everyone will agree that it represents an important measure of the difficulty of the problem. The main contribution of this work is establishing matching (up to constants) upper and lower bounds for the risk \eqref{minimax:risk} for any bounded convex set  $K$. In particular we would like to single out the upper bound as the main contribution, as the lower bound is a simple consequence of Fano's inequality. In order to establish the upper bound, we demonstrate that there exists a universal scheme which attains the minimax rate for any bounded convex set $K$. The existence of such a general scheme should not be a priori obvious, nonetheless we show it does exist. In order to do that we rely on techniques first proposed by \cite{lecam1973convergence, birge1983approximation}. That being said, while our result may be expected from these works, it is important to note that it cannot be directly derived by using any previously known results. In their work, \cite{lecam1973convergence, birge1983approximation} metrize the probability space using the squared Hellinger distance, and their loss function between the estimate and the true parameter is also based on the squared Hellinger distance. For two multivariate Gaussians $N(\nu_1, \sigma^2 \mathbb{I})$ and $N(\nu_2, \sigma^2 \mathbb{I})$ the squared Hellinger distance is given by $1 - \exp\bigg(-\frac{\|\nu_1 - \nu_2\|^2}{8\sigma^2}\bigg)$ \citep{pardo2018statistical}. This is markedly distinct from the Euclidean norm of the mean difference $\|\nu_1 - \nu_2\|$ which is what we use to metrize the problem, and results in a more natural loss function for the Gaussian sequence model. In particular, the squared Hellinger distance behaves like $\frac{\|\nu_1 - \nu_2\|^2}{8 \sigma^2}$ when $\|\nu_1 - \nu_2\|$ is ``small'', but is of constant order when $\|\nu_1 - \nu_2\|$ is ``large''. This difference renders it impossible to use directly previously known results. In addition, the estimators used by \cite{lecam1973convergence, birge1983approximation} are rather involved, and use pairwise testing on Hellinger balls. In contrast the estimator we propose in this work, does not involve such complicated pairwise tests; it does however, involve delicate constructions of packing sets. We would like to be upfront in that in this work we do not propose a fully satisfactory resolution of this problem for any bounded convex set $K$, as our general algorithm, although very simple to state presents substantial implementational challenges, and is not computationally tractable. We further extend our result to the unbounded case with known variance of the noise.

The constrained Gaussian sequence model setting has numerous applications. For instance, in the special case when the set $K$ is an ellipse, \cite{wei2020gauss} show two examples --- one of constrained ridge regression with fixed design, and one of nonparametric regression with reproducing kernels which can both be viewed through the Gaussian sequence model perspective. In addition, functional regression with shape-constraints, such as isotonic regression or convex regression can often be viewed through the sequence model lens  \citep[see, e.g.][and references therein]{bellec2018sharp, guntuboyina2018nonparametric}. In the latter literature often times a preferred estimator is the constrained least squares estimator (LSE), which is known to be minimax optimal in some settings. Additional examples of how the Gaussian sequence model encompasses different models are given in \cite{chatterjee2014new}, where the author illustrates how both constrained LASSO with fixed design and isotonic regression can be thought of as sequence models under convex constraints. He also shows that unfortunately the LSE is not minimax optimal in general, as there exist convex sets where the gap between the minimax rate and the performance of the LSE can be as large as $\sqrt{n}$ (on the squared risk scale when $\sigma = 1$). This counterexample naturally leads \cite{chatterjee2014new} to ask the question ``as to whether there is a general estimator that is guaranteed to be minimax up to a universal constant''. Hence the need arises to find other estimators which always enjoy minimaxity.



\subsection{Related Literature}

There is a tremendous amount of work on the Gaussian sequence model. Here we will only scratch the surface. The interested reader can consult with books on the sequence model and nonparametric statistics such as \cite{johnstone2011gaussian, nemirovski1998lectures, tsybakov2009introduction}.

In one of the most classical results, \cite{pinsker1980optimal} showed the precise linear minimax rate when the set $K$ is an ellipse, and in fact he showed that a linear estimate achieves the minimax rate when $\sigma \rightarrow 0$. Pinsker's results are valid in a framework more general than the one we consider in this paper as he looked at ellipses in the $\ell_2$ space, whereas we consider only subsets of $\mathbb{R}^n$. When $n = 1$ any bounded convex set is an interval and in that sense the works of \cite{casella1981estimating, bickel1981minimax, ibragimov1985nonparametric} are very relevant. We will later see when we consider the example of hyperrectangles that we are able to recover their result up to constant factors. In a classic work, \cite{donoho1990minimax} consider almost the exact same problem as we consider here (with $\ell_2$ instead of $\mathbb{R}^n$) and work out a variety of special cases for $K$ --- such as hyperrectangles, ellipses, and orthosymmetric quadratically convex sets. They show that a linear projection estimator (also known as the truncated series estimator) is minimax optimal up to constants in all of these examples. We will re-derive all of their results (up to constants) in the Examples section to follow. \cite{javanmard2012minimax} derive the minimax rate for symmetric convex polytopes up to logarithmic factors using the truncated series estimator. \cite{javanmard2012minimax} also point out in their introduction, that ``it is still largely unkown how to compute the minimax risk for an arbitrary convex body''. \cite{zhang2013nearly} obtains the minimax rate up to a logarithmic factor for $\ell_q$ balls for $q \leq 1$, by using an estimator which is a mixture of LSE and a linear projection estimator. \cite{chen2017note} extend results of \cite{chatterjee2014new} to show that the LSE and other regularized estimators are admissible up to universal constants in the same setting that we consider. We will see later on that our estimator, although of different nature than the aforementioned ones, also has this property due to the fact that it is minimax up to constant factors. In a recent paper, \cite{ermakov2020minimax} shows that the linear minimax risk in the sequence model in $\ell_2$ can be explicitly quantified for certain convex sets of the form $K = \{x = \{x_i\}_{i = 1}^{\infty} : \sup_k a_k^{-1} \sum_{j = k}^{\infty} x_j^2 \leq P_0\}$ with $a_k > 0$ being a decreasing sequence. Moreover, \cite{ermakov2020minimax} shows that the asymptotic minimax risk when $a_k = k^{-2\alpha}$ can be precisely quantified as well.\\
\indent Aside from the aforementioned works which focus on the Gaussian sequence model, we would like to discuss the celebrated paper of \cite{yang1999information} which is also highly relevant (yet does not consider the sequence model per se). \cite{yang1999information} based their work on the premise that local entropy is hard to calculate in general, yet it had been shown that it leads to optimal rates of convergence by \cite{lecam1973convergence, birge1983approximation} in certain problems metrized with the squared Hellinger distance as we alluded to previously. Therefore \cite{yang1999information} proposed to study the global entropy instead, which is often easier to handle. We must agree, that local entropy (see Definition \ref{local:entropy:def}) is a challenging quantity to work with, nevertheless, as our result shows it is precisely what is needed to calculate in order to determine the minimax rate for a general convex set $K$. This is also easy to explain intuitively at this point of the paper even without going into the mathematical details. Consider, e.g., the case where the set $K$ is unbounded, e.g., $K$ is a subspace (which corresponds to the linear regression setting). The global entropy of such a set is not even defined (as one cannot pack an unbounded set), yet its local entropy is well defined and calculable. 
We would also further comment that for some sets $K$ it is sufficient to calculate the global entropy as it is of the same order as the local entropy. In fact, \cite{yang1999information} offer a result (see Lemma 3 in Section 7 therein), which connects the local and global entropies. Sometimes, the order of the two quantities coincides, in which case one may resort to calculating the global entropy of $K$ instead. See also Subsection \ref{yang:barron:section:exmaple} where we illustrate this by considering the example of an $\ell_1$ ball. 

\subsection{Organization}

The paper is structured as follows. We present our main results on bounded convex sets $K$ in Section \ref{main:results:section}. Section \ref{examples:section} is dedicated to some examples. Section \ref{adaptivity:sec} argues that the estimator defined in Section \ref{main:results:section} is adaptive to the true point, and it also is admissible up to a universal constant. Section \ref{unbounded:sets:section} extends our main results from the bounded case to the unbounded case with known $\sigma^2$. A brief discussion is given in Section \ref{discussion:section}.

\subsection{Notation}

We outline some commonly used notation here. We use $\vee$ and $\wedge$ for $\max$ and $\min$ of two numbers respectively. Throughout the paper $\|\cdot\|$ denotes the Euclidean norm. Constants may change values from line to line. For an integer $m\in \mathbb{N}$ we use the shorthand $[m] = \{1,\ldots,m\}$. We use $B(\theta,r)$ to denote a closed Euclidean ball centered at the point $\theta$ with radius $r$. We use $\lesssim$ and $\gtrsim$ to mean $\leq$ and $\geq$ up to absolute constant factors, and for two sequences $a_n$ and $b_n$ we write $a_n \asymp b_n$ if both $a_n \lesssim b_n$ and $a_n \gtrsim b_n$ hold. Throughout the paper we use $\log$ to denote the natural logarithm. 


\section{Main Results} \label{main:results:section}

Here we focus on the following problem. We observe $n$ observations (i.e., $i \in [n]$) $Y_i = \mu_i + \xi_i$, where $\mu \in K$, for $K$ being a bounded convex set and $\xi_i \sim N(0,\sigma^2)$ are i.i.d. random variables. We begin with showing a lower bound.  

\subsection{Lower Bound}

In this subsection we present our main lower bound. It is a simple consequence of Fano's inequality, which we state below for the convenience of the reader. Throughout this section and the rest of the paper $c > 0$ is some sufficiently large absolute constant.

\begin{lemma}[Fano's inequality] Let $\mu^1, \ldots, \mu^m$ be a collection of $\varepsilon$-separated points in the parameter space in Euclidean norm. Suppose $J$ is uniformly distributed over the index set $[m]$, and $(Y | J = j) = \mu^j + \xi$ for $\xi \sim N(0, \mathbb{I}\sigma^2)$. Then 
\begin{align*}
    \inf_{\hat \nu} \sup_{\mu} \mathbb{E} \|\hat \nu(Y) - \mu\|^2 \geq \frac{\varepsilon^2}{4}\bigg(1 - \frac{I(Y; J) + \log 2}{\log m}\bigg).
\end{align*}

\end{lemma}

In the above $I(Y;J)$ is the mutual information between $Y$ and $J$, and can be upper bounded by $\frac{1}{m} \sum_{j} D_{KL}(\mathbb{P}_{\mu^j} || \mathbb{P}_{\nu}) = \frac{1}{m} \sum_{j} \frac{\|\mu^j - \nu\|^2}{2\sigma^2} \leq \max_j \frac{\|\mu^j - \nu\|^2}{2\sigma^2}$ for any $\nu \in \RR^n$ (see (15.52) \cite{wainwright2019high} e.g.). We will now define local packing entropy.

\begin{definition}[Local Entropy]\label{local:entropy:def} Let $\theta \in K$ be a point. Consider the set $B(\theta, \varepsilon) \cap K$. Let $M(\varepsilon/c, B(\theta, \varepsilon) \cap K)$ denote the largest cardinality of an $\varepsilon/c$ packing set \citep[see Defintion 5.4][e.g., for a definition of a packing set]{wainwright2019high} in $B(\theta, \varepsilon) \cap K$. Let \begin{align*}
    M^{\operatorname{loc}}(\varepsilon) = \sup_{\theta \in K} M(\varepsilon/c, B(\theta, \varepsilon) \cap K).
\end{align*} We refer to $\log M^{\operatorname{loc}}(\varepsilon)$ as local entropy of $K$. Sometimes we will use $M^{\operatorname{loc}}_K(\varepsilon)$ if we the set $K$ is not clear from the context.
\end{definition}

\begin{lemma} We have
\begin{align*}
     \inf_{\hat \nu} \sup_{\mu} \mathbb{E} \|\hat \nu(Y) - \mu\|^2 \geq \frac{\varepsilon^2}{8 c^2},
\end{align*}
for any $\varepsilon$ satisfying $\log M^{\operatorname{loc}}(\varepsilon) > 4(\varepsilon^2/(2\sigma^2) \vee \log 2)$, where $c$ is the constant from the Definition \ref{local:entropy:def} which is fixed to some large enough value.
\end{lemma}

\begin{proof}
For a given $\varepsilon$ we can build an $\varepsilon/c$-local packing of cardinality  $M^{\operatorname{loc}}(\varepsilon)$,  around some point of $K$. If such a point does not exist, we can take a sequence of points which achieve this in the limit, which is good enough for our argument to follow. Suppose that $\log M^{\operatorname{loc}}(\varepsilon) > 2(\varepsilon^2/(2\sigma^2) + \log 2)$. From Fano's inequality it immediately follows that the minimax risk is at least $\frac{\varepsilon^2}{8c^2}$. The above is implied when $\log M^{\operatorname{loc}}(\varepsilon) > 4(\varepsilon^2/(2\sigma^2) \vee \log 2)$.
\end{proof}

\subsection{Upper Bound}\label{upper:bound:section}

In this subsection we focus on the upper bound. Let $d = \operatorname{diam}(K)$. We propose the estimator described in Algorithm \ref{test}, where $2 (C + 1) = c$ is the constant from the definition of local entropy which is assumed to be sufficiently large. The reader will notice that our algorithm contains an infinite loop. This means that our estimator can only be achieved in theory. The good news is that if one knows a lower bound on $\sigma$ (including cases when one knows $\sigma$ exactly), one need not run the procedure ad infinitum. In that case the number of iterations can be determined through a concentration result to follow. We give an updated algorithm with finitely many iterations and additional details of this in Appendix \ref{appendix:A}.

In order to ease the reader into Algorithm \ref{test}, we also outline in plain English how the first few iterations work. For simplicity we will describe the algorithm as if the packing sets are selected during the estimation process, but they should be constructed prior to seeing the data. At first we select an arbitrary point $\nu^* \in K$. Then we consider the ball $B(\nu^*, d) \cap K = K$, and we take a maximal packing set at $\frac{d}{2(C+1)} = \frac{d}{c}$ distance. Let $M_1$ denote the corresponding maximal packing set. Reassign $\nu^*$ to be the closest point to $Y$ from the set $M_1$ in Euclidean distance, i.e., let $\nu^* = \argmin_{\nu \in M_1} \|Y - \nu\|$. Consider the set $B(\nu^*, d/2)\cap K$ and its maximal packing set at a distance $\frac{d}{4(C+1)} = \frac{d}{2c}$ and call it $M_2$. Once again reassign $\nu^* = \argmin_{\nu \in M_2} \|Y - \nu\|$. For the next step consider the set $B(\nu^*, d/4)\cap K$ and its maximal packing at a distance $\frac{d}{8(C+1)} = \frac{d}{4c}$ and call it $M_3$. Reassign $\nu^* = \argmin_{\nu \in M_3} \|Y - \nu\|$. Figure \ref{fig:diagram} illustrates these three steps. Continue the process and output the limiting point.

Before we proceed, we pause to observe a quick fact about the packing sets that are introduced in Algorithm \ref{test}. It is simple to see that if one takes the union of all points from the packing sets on all levels, these points form a countable dense subset of $\overline K$ which is the closure of $K$, and hence any point in $\overline K$ is potentially achievable in the limit. This means that if $K$ is not closed our estimator may not be proper (i.e., it may output points outside of $K$, but the estimator will always be a limiting point of points in $K$ at worst). Furthermore, as we will see later (see the proof of Lemma \ref{compact:set:estimators}) if the point $Y \in \overline K$, Algorithm \ref{test} will always output the point $Y$. The latter is clearly a desirable property, since when $\sigma = 0$, one needs to pick the observed point to achieve minimaxity, and our estimator is not given knowledge of $\sigma$.

\begin{figure}[ht!]
    \centering
\tikzset{every picture/.style={line width=0.75pt}} 

\begin{tikzpicture}[x=0.65pt,y=0.65pt,yscale=-1,xscale=1]

\draw  [color={rgb, 255:red, 0; green, 0; blue, 0 }  ,draw opacity=1 ][fill={rgb, 255:red, 0; green, 0; blue, 0 }  ,fill opacity=1 ] (245.52,289.76) .. controls (245.52,288.66) and (246.42,287.76) .. (247.52,287.76) .. controls (248.62,287.76) and (249.52,288.66) .. (249.52,289.76) .. controls (249.52,290.87) and (248.62,291.76) .. (247.52,291.76) .. controls (246.42,291.76) and (245.52,290.87) .. (245.52,289.76) -- cycle ;
\draw  [color={rgb, 255:red, 208; green, 2; blue, 27 }  ,draw opacity=1 ] (127.1,274.41) .. controls (115.26,260.26) and (134.51,224.67) .. (170.08,194.9) .. controls (205.66,165.13) and (244.09,152.47) .. (255.93,166.61) .. controls (267.76,180.76) and (248.52,216.35) .. (212.94,246.12) .. controls (177.37,275.89) and (138.93,288.55) .. (127.1,274.41) -- cycle ;
\draw  [color={rgb, 255:red, 0; green, 0; blue, 124 }  ,draw opacity=1 ][dash pattern={on 5.63pt off 4.5pt}][line width=1.5]  (19.38,222.32) .. controls (19.38,128.26) and (95.63,52) .. (189.7,52) .. controls (283.77,52) and (360.02,128.26) .. (360.02,222.32) .. controls (360.02,316.39) and (283.77,392.64) .. (189.7,392.64) .. controls (95.63,392.64) and (19.38,316.39) .. (19.38,222.32) -- cycle ;
\draw [color={rgb, 255:red, 0; green, 0; blue, 124 }  ,draw opacity=1 ][line width=1.5]  [dash pattern={on 1.69pt off 2.76pt}]  (158.5,55) -- (191.51,220.51) ;
\draw  [color={rgb, 255:red, 0; green, 0; blue, 124 }  ,draw opacity=1 ][fill={rgb, 255:red, 0; green, 0; blue, 124 }  ,fill opacity=1 ] (189.51,220.51) .. controls (189.51,219.41) and (190.41,218.51) .. (191.51,218.51) .. controls (192.62,218.51) and (193.51,219.41) .. (193.51,220.51) .. controls (193.51,221.61) and (192.62,222.51) .. (191.51,222.51) .. controls (190.41,222.51) and (189.51,221.61) .. (189.51,220.51) -- cycle ;
\draw  [color={rgb, 255:red, 208; green, 2; blue, 27 }  ,draw opacity=1 ][fill={rgb, 255:red, 208; green, 2; blue, 27 }  ,fill opacity=1 ] (224.52,215.76) .. controls (224.52,214.66) and (225.42,213.76) .. (226.52,213.76) .. controls (227.62,213.76) and (228.52,214.66) .. (228.52,215.76) .. controls (228.52,216.87) and (227.62,217.76) .. (226.52,217.76) .. controls (225.42,217.76) and (224.52,216.87) .. (224.52,215.76) -- cycle ;
\draw  [color={rgb, 255:red, 208; green, 2; blue, 27 }  ,draw opacity=1 ][fill={rgb, 255:red, 208; green, 2; blue, 27 }  ,fill opacity=1 ] (241.5,189) .. controls (241.5,187.9) and (242.4,187) .. (243.5,187) .. controls (244.6,187) and (245.5,187.9) .. (245.5,189) .. controls (245.5,190.1) and (244.6,191) .. (243.5,191) .. controls (242.4,191) and (241.5,190.1) .. (241.5,189) -- cycle ;
\draw [color={rgb, 255:red, 208; green, 2; blue, 27 }  ,draw opacity=1 ][line width=1.5]  [dash pattern={on 1.69pt off 2.76pt}]  (164.52,220.76) -- (152.5,245) ;
\draw  [color={rgb, 255:red, 208; green, 2; blue, 27 }  ,draw opacity=1 ][fill={rgb, 255:red, 208; green, 2; blue, 27 }  ,fill opacity=1 ] (212.52,189.76) .. controls (212.52,188.66) and (213.42,187.76) .. (214.52,187.76) .. controls (215.62,187.76) and (216.52,188.66) .. (216.52,189.76) .. controls (216.52,190.87) and (215.62,191.76) .. (214.52,191.76) .. controls (213.42,191.76) and (212.52,190.87) .. (212.52,189.76) -- cycle ;
\draw  [color={rgb, 255:red, 208; green, 2; blue, 27 }  ,draw opacity=1 ][fill={rgb, 255:red, 208; green, 2; blue, 27 }  ,fill opacity=1 ] (186.52,204.76) .. controls (186.52,203.66) and (187.42,202.76) .. (188.52,202.76) .. controls (189.62,202.76) and (190.52,203.66) .. (190.52,204.76) .. controls (190.52,205.87) and (189.62,206.76) .. (188.52,206.76) .. controls (187.42,206.76) and (186.52,205.87) .. (186.52,204.76) -- cycle ;
\draw  [color={rgb, 255:red, 208; green, 2; blue, 27 }  ,draw opacity=1 ][fill={rgb, 255:red, 208; green, 2; blue, 27 }  ,fill opacity=1 ] (162.52,218.76) .. controls (162.52,217.66) and (163.42,216.76) .. (164.52,216.76) .. controls (165.62,216.76) and (166.52,217.66) .. (166.52,218.76) .. controls (166.52,219.87) and (165.62,220.76) .. (164.52,220.76) .. controls (163.42,220.76) and (162.52,219.87) .. (162.52,218.76) -- cycle ;
\draw  [color={rgb, 255:red, 208; green, 2; blue, 27 }  ,draw opacity=1 ][fill={rgb, 255:red, 208; green, 2; blue, 27 }  ,fill opacity=1 ] (151.52,244.76) .. controls (151.52,243.66) and (152.42,242.76) .. (153.52,242.76) .. controls (154.62,242.76) and (155.52,243.66) .. (155.52,244.76) .. controls (155.52,245.87) and (154.62,246.76) .. (153.52,246.76) .. controls (152.42,246.76) and (151.52,245.87) .. (151.52,244.76) -- cycle ;
\draw  [color={rgb, 255:red, 208; green, 2; blue, 27 }  ,draw opacity=1 ][fill={rgb, 255:red, 208; green, 2; blue, 27 }  ,fill opacity=1 ] (197.52,236.76) .. controls (197.52,235.66) and (198.42,234.76) .. (199.52,234.76) .. controls (200.62,234.76) and (201.52,235.66) .. (201.52,236.76) .. controls (201.52,237.87) and (200.62,238.76) .. (199.52,238.76) .. controls (198.42,238.76) and (197.52,237.87) .. (197.52,236.76) -- cycle ;
\draw  [color={rgb, 255:red, 208; green, 2; blue, 27 }  ,draw opacity=1 ][fill={rgb, 255:red, 208; green, 2; blue, 27 }  ,fill opacity=1 ] (139.52,271.76) .. controls (139.52,270.66) and (140.42,269.76) .. (141.52,269.76) .. controls (142.62,269.76) and (143.52,270.66) .. (143.52,271.76) .. controls (143.52,272.87) and (142.62,273.76) .. (141.52,273.76) .. controls (140.42,273.76) and (139.52,272.87) .. (139.52,271.76) -- cycle ;
\draw  [color={rgb, 255:red, 208; green, 2; blue, 27 }  ,draw opacity=1 ] (456.1,158.41) .. controls (444.26,144.26) and (463.51,108.67) .. (499.08,78.9) .. controls (534.66,49.13) and (573.09,36.47) .. (584.93,50.61) .. controls (596.76,64.76) and (577.52,100.35) .. (541.94,130.12) .. controls (506.37,159.89) and (467.93,172.55) .. (456.1,158.41) -- cycle ;
\draw  [color={rgb, 255:red, 0; green, 0; blue, 124 }  ,draw opacity=1 ][dash pattern={on 5.63pt off 4.5pt}][line width=1.5]  (442.54,122.37) .. controls (442.54,75.77) and (480.32,38) .. (526.91,38) .. controls (573.51,38) and (611.28,75.77) .. (611.28,122.37) .. controls (611.28,168.97) and (573.51,206.74) .. (526.91,206.74) .. controls (480.32,206.74) and (442.54,168.97) .. (442.54,122.37) -- cycle ;
\draw [color={rgb, 255:red, 0; green, 0; blue, 124 }  ,draw opacity=1 ][line width=1.5]  [dash pattern={on 1.69pt off 2.76pt}]  (510.5,40) -- (528.52,120.76) ;
\draw  [color={rgb, 255:red, 208; green, 2; blue, 27 }  ,draw opacity=1 ][fill={rgb, 255:red, 208; green, 2; blue, 27 }  ,fill opacity=1 ] (562.52,67.76) .. controls (562.52,66.66) and (563.42,65.76) .. (564.52,65.76) .. controls (565.62,65.76) and (566.52,66.66) .. (566.52,67.76) .. controls (566.52,68.87) and (565.62,69.76) .. (564.52,69.76) .. controls (563.42,69.76) and (562.52,68.87) .. (562.52,67.76) -- cycle ;
\draw  [color={rgb, 255:red, 208; green, 2; blue, 27 }  ,draw opacity=1 ][fill={rgb, 255:red, 208; green, 2; blue, 27 }  ,fill opacity=1 ] (554.5,77) .. controls (554.5,75.9) and (555.4,75) .. (556.5,75) .. controls (557.6,75) and (558.5,75.9) .. (558.5,77) .. controls (558.5,78.1) and (557.6,79) .. (556.5,79) .. controls (555.4,79) and (554.5,78.1) .. (554.5,77) -- cycle ;
\draw [color={rgb, 255:red, 208; green, 2; blue, 27 }  ,draw opacity=1 ][line width=1.5]  [dash pattern={on 1.69pt off 2.76pt}]  (493.52,104.76) -- (487.51,120.88) ;
\draw  [color={rgb, 255:red, 0; green, 0; blue, 0 }  ,draw opacity=1 ][fill={rgb, 255:red, 0; green, 0; blue, 0 }  ,fill opacity=1 ] (574.52,172.76) .. controls (574.52,171.66) and (575.42,170.76) .. (576.52,170.76) .. controls (577.62,170.76) and (578.52,171.66) .. (578.52,172.76) .. controls (578.52,173.87) and (577.62,174.76) .. (576.52,174.76) .. controls (575.42,174.76) and (574.52,173.87) .. (574.52,172.76) -- cycle ;
\draw  [color={rgb, 255:red, 208; green, 2; blue, 27 }  ,draw opacity=1 ][fill={rgb, 255:red, 208; green, 2; blue, 27 }  ,fill opacity=1 ] (534.52,71.76) .. controls (534.52,70.66) and (535.42,69.76) .. (536.52,69.76) .. controls (537.62,69.76) and (538.52,70.66) .. (538.52,71.76) .. controls (538.52,72.87) and (537.62,73.76) .. (536.52,73.76) .. controls (535.42,73.76) and (534.52,72.87) .. (534.52,71.76) -- cycle ;
\draw  [color={rgb, 255:red, 208; green, 2; blue, 27 }  ,draw opacity=1 ][fill={rgb, 255:red, 208; green, 2; blue, 27 }  ,fill opacity=1 ] (506.52,90.76) .. controls (506.52,89.66) and (507.42,88.76) .. (508.52,88.76) .. controls (509.62,88.76) and (510.52,89.66) .. (510.52,90.76) .. controls (510.52,91.87) and (509.62,92.76) .. (508.52,92.76) .. controls (507.42,92.76) and (506.52,91.87) .. (506.52,90.76) -- cycle ;
\draw  [color={rgb, 255:red, 208; green, 2; blue, 27 }  ,draw opacity=1 ][fill={rgb, 255:red, 208; green, 2; blue, 27 }  ,fill opacity=1 ] (491.52,102.76) .. controls (491.52,101.66) and (492.42,100.76) .. (493.52,100.76) .. controls (494.62,100.76) and (495.52,101.66) .. (495.52,102.76) .. controls (495.52,103.87) and (494.62,104.76) .. (493.52,104.76) .. controls (492.42,104.76) and (491.52,103.87) .. (491.52,102.76) -- cycle ;
\draw  [color={rgb, 255:red, 208; green, 2; blue, 27 }  ,draw opacity=1 ][fill={rgb, 255:red, 208; green, 2; blue, 27 }  ,fill opacity=1 ] (485.51,118.88) .. controls (485.51,117.78) and (486.41,116.88) .. (487.51,116.88) .. controls (488.61,116.88) and (489.51,117.78) .. (489.51,118.88) .. controls (489.51,119.99) and (488.61,120.88) .. (487.51,120.88) .. controls (486.41,120.88) and (485.51,119.99) .. (485.51,118.88) -- cycle ;
\draw  [color={rgb, 255:red, 0; green, 0; blue, 124 }  ,draw opacity=1 ][fill={rgb, 255:red, 0; green, 0; blue, 124 }  ,fill opacity=1 ] (526.52,120.76) .. controls (526.52,119.66) and (527.42,118.76) .. (528.52,118.76) .. controls (529.62,118.76) and (530.52,119.66) .. (530.52,120.76) .. controls (530.52,121.87) and (529.62,122.76) .. (528.52,122.76) .. controls (527.42,122.76) and (526.52,121.87) .. (526.52,120.76) -- cycle ;
\draw  [color={rgb, 255:red, 208; green, 2; blue, 27 }  ,draw opacity=1 ][fill={rgb, 255:red, 208; green, 2; blue, 27 }  ,fill opacity=1 ] (468.52,155.76) .. controls (468.52,154.66) and (469.42,153.76) .. (470.52,153.76) .. controls (471.62,153.76) and (472.52,154.66) .. (472.52,155.76) .. controls (472.52,156.87) and (471.62,157.76) .. (470.52,157.76) .. controls (469.42,157.76) and (468.52,156.87) .. (468.52,155.76) -- cycle ;
\draw  [color={rgb, 255:red, 208; green, 2; blue, 27 }  ,draw opacity=1 ][fill={rgb, 255:red, 208; green, 2; blue, 27 }  ,fill opacity=1 ] (550.5,62) .. controls (550.5,60.9) and (551.4,60) .. (552.5,60) .. controls (553.6,60) and (554.5,60.9) .. (554.5,62) .. controls (554.5,63.1) and (553.6,64) .. (552.5,64) .. controls (551.4,64) and (550.5,63.1) .. (550.5,62) -- cycle ;
\draw  [color={rgb, 255:red, 208; green, 2; blue, 27 }  ,draw opacity=1 ][fill={rgb, 255:red, 208; green, 2; blue, 27 }  ,fill opacity=1 ] (564.52,78.76) .. controls (564.52,77.66) and (565.42,76.76) .. (566.52,76.76) .. controls (567.62,76.76) and (568.52,77.66) .. (568.52,78.76) .. controls (568.52,79.87) and (567.62,80.76) .. (566.52,80.76) .. controls (565.42,80.76) and (564.52,79.87) .. (564.52,78.76) -- cycle ;
\draw  [color={rgb, 255:red, 208; green, 2; blue, 27 }  ,draw opacity=1 ][fill={rgb, 255:red, 208; green, 2; blue, 27 }  ,fill opacity=1 ] (554.52,92.76) .. controls (554.52,91.66) and (555.42,90.76) .. (556.52,90.76) .. controls (557.62,90.76) and (558.52,91.66) .. (558.52,92.76) .. controls (558.52,93.87) and (557.62,94.76) .. (556.52,94.76) .. controls (555.42,94.76) and (554.52,93.87) .. (554.52,92.76) -- cycle ;
\draw  [color={rgb, 255:red, 208; green, 2; blue, 27 }  ,draw opacity=1 ][fill={rgb, 255:red, 208; green, 2; blue, 27 }  ,fill opacity=1 ] (545.52,111.76) .. controls (545.52,110.66) and (546.42,109.76) .. (547.52,109.76) .. controls (548.62,109.76) and (549.52,110.66) .. (549.52,111.76) .. controls (549.52,112.87) and (548.62,113.76) .. (547.52,113.76) .. controls (546.42,113.76) and (545.52,112.87) .. (545.52,111.76) -- cycle ;
\draw  [color={rgb, 255:red, 208; green, 2; blue, 27 }  ,draw opacity=1 ][fill={rgb, 255:red, 208; green, 2; blue, 27 }  ,fill opacity=1 ] (505.52,103.76) .. controls (505.52,102.66) and (506.42,101.76) .. (507.52,101.76) .. controls (508.62,101.76) and (509.52,102.66) .. (509.52,103.76) .. controls (509.52,104.87) and (508.62,105.76) .. (507.52,105.76) .. controls (506.42,105.76) and (505.52,104.87) .. (505.52,103.76) -- cycle ;
\draw  [color={rgb, 255:red, 208; green, 2; blue, 27 }  ,draw opacity=1 ][fill={rgb, 255:red, 208; green, 2; blue, 27 }  ,fill opacity=1 ] (504.52,125.76) .. controls (504.52,124.66) and (505.42,123.76) .. (506.52,123.76) .. controls (507.62,123.76) and (508.52,124.66) .. (508.52,125.76) .. controls (508.52,126.87) and (507.62,127.76) .. (506.52,127.76) .. controls (505.42,127.76) and (504.52,126.87) .. (504.52,125.76) -- cycle ;
\draw  [color={rgb, 255:red, 208; green, 2; blue, 27 }  ,draw opacity=1 ][fill={rgb, 255:red, 208; green, 2; blue, 27 }  ,fill opacity=1 ] (519.52,122.76) .. controls (519.52,121.66) and (520.42,120.76) .. (521.52,120.76) .. controls (522.62,120.76) and (523.52,121.66) .. (523.52,122.76) .. controls (523.52,123.87) and (522.62,124.76) .. (521.52,124.76) .. controls (520.42,124.76) and (519.52,123.87) .. (519.52,122.76) -- cycle ;
\draw  [color={rgb, 255:red, 208; green, 2; blue, 27 }  ,draw opacity=1 ][fill={rgb, 255:red, 208; green, 2; blue, 27 }  ,fill opacity=1 ] (470.52,133.76) .. controls (470.52,132.66) and (471.42,131.76) .. (472.52,131.76) .. controls (473.62,131.76) and (474.52,132.66) .. (474.52,133.76) .. controls (474.52,134.87) and (473.62,135.76) .. (472.52,135.76) .. controls (471.42,135.76) and (470.52,134.87) .. (470.52,133.76) -- cycle ;
\draw  [color={rgb, 255:red, 208; green, 2; blue, 27 }  ,draw opacity=1 ][fill={rgb, 255:red, 208; green, 2; blue, 27 }  ,fill opacity=1 ] (489.52,136.76) .. controls (489.52,135.66) and (490.42,134.76) .. (491.52,134.76) .. controls (492.62,134.76) and (493.52,135.66) .. (493.52,136.76) .. controls (493.52,137.87) and (492.62,138.76) .. (491.52,138.76) .. controls (490.42,138.76) and (489.52,137.87) .. (489.52,136.76) -- cycle ;
\draw  [color={rgb, 255:red, 208; green, 2; blue, 27 }  ,draw opacity=1 ][fill={rgb, 255:red, 208; green, 2; blue, 27 }  ,fill opacity=1 ] (509.52,139.76) .. controls (509.52,138.66) and (510.42,137.76) .. (511.52,137.76) .. controls (512.62,137.76) and (513.52,138.66) .. (513.52,139.76) .. controls (513.52,140.87) and (512.62,141.76) .. (511.52,141.76) .. controls (510.42,141.76) and (509.52,140.87) .. (509.52,139.76) -- cycle ;
\draw  [color={rgb, 255:red, 208; green, 2; blue, 27 }  ,draw opacity=1 ][fill={rgb, 255:red, 208; green, 2; blue, 27 }  ,fill opacity=1 ] (519.52,71.76) .. controls (519.52,70.66) and (520.42,69.76) .. (521.52,69.76) .. controls (522.62,69.76) and (523.52,70.66) .. (523.52,71.76) .. controls (523.52,72.87) and (522.62,73.76) .. (521.52,73.76) .. controls (520.42,73.76) and (519.52,72.87) .. (519.52,71.76) -- cycle ;
\draw  [color={rgb, 255:red, 208; green, 2; blue, 27 }  ,draw opacity=1 ][fill={rgb, 255:red, 208; green, 2; blue, 27 }  ,fill opacity=1 ] (485.52,149.76) .. controls (485.52,148.66) and (486.42,147.76) .. (487.52,147.76) .. controls (488.62,147.76) and (489.52,148.66) .. (489.52,149.76) .. controls (489.52,150.87) and (488.62,151.76) .. (487.52,151.76) .. controls (486.42,151.76) and (485.52,150.87) .. (485.52,149.76) -- cycle ;
\draw  [color={rgb, 255:red, 208; green, 2; blue, 27 }  ,draw opacity=1 ][fill={rgb, 255:red, 208; green, 2; blue, 27 }  ,fill opacity=1 ] (473.52,114.76) .. controls (473.52,113.66) and (474.42,112.76) .. (475.52,112.76) .. controls (476.62,112.76) and (477.52,113.66) .. (477.52,114.76) .. controls (477.52,115.87) and (476.62,116.76) .. (475.52,116.76) .. controls (474.42,116.76) and (473.52,115.87) .. (473.52,114.76) -- cycle ;
\draw  [color={rgb, 255:red, 208; green, 2; blue, 27 }  ,draw opacity=1 ][fill={rgb, 255:red, 208; green, 2; blue, 27 }  ,fill opacity=1 ] (536.52,98.76) .. controls (536.52,97.66) and (537.42,96.76) .. (538.52,96.76) .. controls (539.62,96.76) and (540.52,97.66) .. (540.52,98.76) .. controls (540.52,99.87) and (539.62,100.76) .. (538.52,100.76) .. controls (537.42,100.76) and (536.52,99.87) .. (536.52,98.76) -- cycle ;
\draw [color={rgb, 255:red, 0; green, 0; blue, 124 }  ,draw opacity=1 ]   (487,76) .. controls (481.15,84.78) and (497.16,88.8) .. (517.91,81.01) ;
\draw [shift={(519.51,80.38)}, rotate = 158.17] [color={rgb, 255:red, 0; green, 0; blue, 124 }  ,draw opacity=1 ][line width=0.75]    (10.93,-3.29) .. controls (6.95,-1.4) and (3.31,-0.3) .. (0,0) .. controls (3.31,0.3) and (6.95,1.4) .. (10.93,3.29)   ;
\draw  [color={rgb, 255:red, 208; green, 2; blue, 27 }  ,draw opacity=1 ][fill={rgb, 255:red, 208; green, 2; blue, 27 }  ,fill opacity=1 ] (529.52,131.76) .. controls (529.52,130.66) and (530.42,129.76) .. (531.52,129.76) .. controls (532.62,129.76) and (533.52,130.66) .. (533.52,131.76) .. controls (533.52,132.87) and (532.62,133.76) .. (531.52,133.76) .. controls (530.42,133.76) and (529.52,132.87) .. (529.52,131.76) -- cycle ;
\draw [color={rgb, 255:red, 208; green, 2; blue, 27 }  ,draw opacity=1 ]   (461.01,100.38) .. controls (455.22,109.07) and (471.3,106.22) .. (488.61,109.61) ;
\draw [shift={(490.5,110)}, rotate = 192.53] [color={rgb, 255:red, 208; green, 2; blue, 27 }  ,draw opacity=1 ][line width=0.75]    (10.93,-3.29) .. controls (6.95,-1.4) and (3.31,-0.3) .. (0,0) .. controls (3.31,0.3) and (6.95,1.4) .. (10.93,3.29)   ;
\draw  [color={rgb, 255:red, 208; green, 2; blue, 27 }  ,draw opacity=1 ] (456.1,381.41) .. controls (444.26,367.26) and (463.51,331.67) .. (499.08,301.9) .. controls (534.66,272.13) and (573.09,259.47) .. (584.93,273.61) .. controls (596.76,287.76) and (577.52,323.35) .. (541.94,353.12) .. controls (506.37,382.89) and (467.93,395.55) .. (456.1,381.41) -- cycle ;
\draw  [color={rgb, 255:red, 0; green, 0; blue, 124 }  ,draw opacity=1 ][dash pattern={on 5.63pt off 4.5pt}][line width=1.5]  (488.24,354.76) .. controls (488.24,330.86) and (507.61,311.48) .. (531.52,311.48) .. controls (555.42,311.48) and (574.8,330.86) .. (574.8,354.76) .. controls (574.8,378.67) and (555.42,398.05) .. (531.52,398.05) .. controls (507.61,398.05) and (488.24,378.67) .. (488.24,354.76) -- cycle ;
\draw [color={rgb, 255:red, 0; green, 0; blue, 124 }  ,draw opacity=1 ][line width=1.5]  [dash pattern={on 1.69pt off 2.76pt}]  (523,314) -- (531.52,354.76) ;
\draw [color={rgb, 255:red, 208; green, 2; blue, 27 }  ,draw opacity=1 ][line width=1.5]  [dash pattern={on 1.69pt off 2.76pt}]  (518.52,356.76) -- (513.52,362.76) ;
\draw  [color={rgb, 255:red, 0; green, 0; blue, 0 }  ,draw opacity=1 ][fill={rgb, 255:red, 0; green, 0; blue, 0 }  ,fill opacity=1 ] (574.52,395.76) .. controls (574.52,394.66) and (575.42,393.76) .. (576.52,393.76) .. controls (577.62,393.76) and (578.52,394.66) .. (578.52,395.76) .. controls (578.52,396.87) and (577.62,397.76) .. (576.52,397.76) .. controls (575.42,397.76) and (574.52,396.87) .. (574.52,395.76) -- cycle ;
\draw  [color={rgb, 255:red, 208; green, 2; blue, 27 }  ,draw opacity=1 ][fill={rgb, 255:red, 208; green, 2; blue, 27 }  ,fill opacity=1 ] (553.52,328.76) .. controls (553.52,327.66) and (554.42,326.76) .. (555.52,326.76) .. controls (556.62,326.76) and (557.52,327.66) .. (557.52,328.76) .. controls (557.52,329.87) and (556.62,330.76) .. (555.52,330.76) .. controls (554.42,330.76) and (553.52,329.87) .. (553.52,328.76) -- cycle ;
\draw  [color={rgb, 255:red, 208; green, 2; blue, 27 }  ,draw opacity=1 ][fill={rgb, 255:red, 208; green, 2; blue, 27 }  ,fill opacity=1 ] (515.52,320.76) .. controls (515.52,319.66) and (516.42,318.76) .. (517.52,318.76) .. controls (518.62,318.76) and (519.52,319.66) .. (519.52,320.76) .. controls (519.52,321.87) and (518.62,322.76) .. (517.52,322.76) .. controls (516.42,322.76) and (515.52,321.87) .. (515.52,320.76) -- cycle ;
\draw  [color={rgb, 255:red, 208; green, 2; blue, 27 }  ,draw opacity=1 ][fill={rgb, 255:red, 208; green, 2; blue, 27 }  ,fill opacity=1 ] (506.52,337.76) .. controls (506.52,336.66) and (507.42,335.76) .. (508.52,335.76) .. controls (509.62,335.76) and (510.52,336.66) .. (510.52,337.76) .. controls (510.52,338.87) and (509.62,339.76) .. (508.52,339.76) .. controls (507.42,339.76) and (506.52,338.87) .. (506.52,337.76) -- cycle ;
\draw  [color={rgb, 255:red, 208; green, 2; blue, 27 }  ,draw opacity=1 ][fill={rgb, 255:red, 208; green, 2; blue, 27 }  ,fill opacity=1 ] (494.51,353.88) .. controls (494.51,352.78) and (495.41,351.88) .. (496.51,351.88) .. controls (497.61,351.88) and (498.51,352.78) .. (498.51,353.88) .. controls (498.51,354.99) and (497.61,355.88) .. (496.51,355.88) .. controls (495.41,355.88) and (494.51,354.99) .. (494.51,353.88) -- cycle ;
\draw  [color={rgb, 255:red, 208; green, 2; blue, 27 }  ,draw opacity=1 ][fill={rgb, 255:red, 208; green, 2; blue, 27 }  ,fill opacity=1 ] (496.52,342.76) .. controls (496.52,341.66) and (497.42,340.76) .. (498.52,340.76) .. controls (499.62,340.76) and (500.52,341.66) .. (500.52,342.76) .. controls (500.52,343.87) and (499.62,344.76) .. (498.52,344.76) .. controls (497.42,344.76) and (496.52,343.87) .. (496.52,342.76) -- cycle ;
\draw  [color={rgb, 255:red, 208; green, 2; blue, 27 }  ,draw opacity=1 ][fill={rgb, 255:red, 208; green, 2; blue, 27 }  ,fill opacity=1 ] (545.52,323.76) .. controls (545.52,322.66) and (546.42,321.76) .. (547.52,321.76) .. controls (548.62,321.76) and (549.52,322.66) .. (549.52,323.76) .. controls (549.52,324.87) and (548.62,325.76) .. (547.52,325.76) .. controls (546.42,325.76) and (545.52,324.87) .. (545.52,323.76) -- cycle ;
\draw  [color={rgb, 255:red, 208; green, 2; blue, 27 }  ,draw opacity=1 ][fill={rgb, 255:red, 208; green, 2; blue, 27 }  ,fill opacity=1 ] (545.52,334.76) .. controls (545.52,333.66) and (546.42,332.76) .. (547.52,332.76) .. controls (548.62,332.76) and (549.52,333.66) .. (549.52,334.76) .. controls (549.52,335.87) and (548.62,336.76) .. (547.52,336.76) .. controls (546.42,336.76) and (545.52,335.87) .. (545.52,334.76) -- cycle ;
\draw  [color={rgb, 255:red, 208; green, 2; blue, 27 }  ,draw opacity=1 ][fill={rgb, 255:red, 208; green, 2; blue, 27 }  ,fill opacity=1 ] (505.52,326.76) .. controls (505.52,325.66) and (506.42,324.76) .. (507.52,324.76) .. controls (508.62,324.76) and (509.52,325.66) .. (509.52,326.76) .. controls (509.52,327.87) and (508.62,328.76) .. (507.52,328.76) .. controls (506.42,328.76) and (505.52,327.87) .. (505.52,326.76) -- cycle ;
\draw  [color={rgb, 255:red, 208; green, 2; blue, 27 }  ,draw opacity=1 ][fill={rgb, 255:red, 208; green, 2; blue, 27 }  ,fill opacity=1 ] (504.52,348.76) .. controls (504.52,347.66) and (505.42,346.76) .. (506.52,346.76) .. controls (507.62,346.76) and (508.52,347.66) .. (508.52,348.76) .. controls (508.52,349.87) and (507.62,350.76) .. (506.52,350.76) .. controls (505.42,350.76) and (504.52,349.87) .. (504.52,348.76) -- cycle ;
\draw  [color={rgb, 255:red, 208; green, 2; blue, 27 }  ,draw opacity=1 ][fill={rgb, 255:red, 208; green, 2; blue, 27 }  ,fill opacity=1 ] (519.52,345.76) .. controls (519.52,344.66) and (520.42,343.76) .. (521.52,343.76) .. controls (522.62,343.76) and (523.52,344.66) .. (523.52,345.76) .. controls (523.52,346.87) and (522.62,347.76) .. (521.52,347.76) .. controls (520.42,347.76) and (519.52,346.87) .. (519.52,345.76) -- cycle ;
\draw  [color={rgb, 255:red, 208; green, 2; blue, 27 }  ,draw opacity=1 ][fill={rgb, 255:red, 208; green, 2; blue, 27 }  ,fill opacity=1 ] (518.52,356.76) .. controls (518.52,355.66) and (519.42,354.76) .. (520.52,354.76) .. controls (521.62,354.76) and (522.52,355.66) .. (522.52,356.76) .. controls (522.52,357.87) and (521.62,358.76) .. (520.52,358.76) .. controls (519.42,358.76) and (518.52,357.87) .. (518.52,356.76) -- cycle ;
\draw  [color={rgb, 255:red, 208; green, 2; blue, 27 }  ,draw opacity=1 ][fill={rgb, 255:red, 208; green, 2; blue, 27 }  ,fill opacity=1 ] (501.52,358.76) .. controls (501.52,357.66) and (502.42,356.76) .. (503.52,356.76) .. controls (504.62,356.76) and (505.52,357.66) .. (505.52,358.76) .. controls (505.52,359.87) and (504.62,360.76) .. (503.52,360.76) .. controls (502.42,360.76) and (501.52,359.87) .. (501.52,358.76) -- cycle ;
\draw  [color={rgb, 255:red, 208; green, 2; blue, 27 }  ,draw opacity=1 ][fill={rgb, 255:red, 208; green, 2; blue, 27 }  ,fill opacity=1 ] (509.52,362.76) .. controls (509.52,361.66) and (510.42,360.76) .. (511.52,360.76) .. controls (512.62,360.76) and (513.52,361.66) .. (513.52,362.76) .. controls (513.52,363.87) and (512.62,364.76) .. (511.52,364.76) .. controls (510.42,364.76) and (509.52,363.87) .. (509.52,362.76) -- cycle ;
\draw  [color={rgb, 255:red, 208; green, 2; blue, 27 }  ,draw opacity=1 ][fill={rgb, 255:red, 208; green, 2; blue, 27 }  ,fill opacity=1 ] (535.52,330.76) .. controls (535.52,329.66) and (536.42,328.76) .. (537.52,328.76) .. controls (538.62,328.76) and (539.52,329.66) .. (539.52,330.76) .. controls (539.52,331.87) and (538.62,332.76) .. (537.52,332.76) .. controls (536.42,332.76) and (535.52,331.87) .. (535.52,330.76) -- cycle ;
\draw  [color={rgb, 255:red, 208; green, 2; blue, 27 }  ,draw opacity=1 ][fill={rgb, 255:red, 208; green, 2; blue, 27 }  ,fill opacity=1 ] (499.52,370.76) .. controls (499.52,369.66) and (500.42,368.76) .. (501.52,368.76) .. controls (502.62,368.76) and (503.52,369.66) .. (503.52,370.76) .. controls (503.52,371.87) and (502.62,372.76) .. (501.52,372.76) .. controls (500.42,372.76) and (499.52,371.87) .. (499.52,370.76) -- cycle ;
\draw  [color={rgb, 255:red, 208; green, 2; blue, 27 }  ,draw opacity=1 ][fill={rgb, 255:red, 208; green, 2; blue, 27 }  ,fill opacity=1 ] (498.52,333.76) .. controls (498.52,332.66) and (499.42,331.76) .. (500.52,331.76) .. controls (501.62,331.76) and (502.52,332.66) .. (502.52,333.76) .. controls (502.52,334.87) and (501.62,335.76) .. (500.52,335.76) .. controls (499.42,335.76) and (498.52,334.87) .. (498.52,333.76) -- cycle ;
\draw  [color={rgb, 255:red, 208; green, 2; blue, 27 }  ,draw opacity=1 ][fill={rgb, 255:red, 208; green, 2; blue, 27 }  ,fill opacity=1 ] (536.52,321.76) .. controls (536.52,320.66) and (537.42,319.76) .. (538.52,319.76) .. controls (539.62,319.76) and (540.52,320.66) .. (540.52,321.76) .. controls (540.52,322.87) and (539.62,323.76) .. (538.52,323.76) .. controls (537.42,323.76) and (536.52,322.87) .. (536.52,321.76) -- cycle ;
\draw [color={rgb, 255:red, 0; green, 0; blue, 124 }  ,draw opacity=1 ]   (488,299) .. controls (482.18,307.73) and (507.41,329.63) .. (526.27,339.15) ;
\draw [shift={(528,340)}, rotate = 205.35] [color={rgb, 255:red, 0; green, 0; blue, 124 }  ,draw opacity=1 ][line width=0.75]    (10.93,-3.29) .. controls (6.95,-1.4) and (3.31,-0.3) .. (0,0) .. controls (3.31,0.3) and (6.95,1.4) .. (10.93,3.29)   ;
\draw  [color={rgb, 255:red, 0; green, 0; blue, 124 }  ,draw opacity=1 ][fill={rgb, 255:red, 0; green, 0; blue, 124 }  ,fill opacity=1 ] (529.52,354.76) .. controls (529.52,353.66) and (530.42,352.76) .. (531.52,352.76) .. controls (532.62,352.76) and (533.52,353.66) .. (533.52,354.76) .. controls (533.52,355.87) and (532.62,356.76) .. (531.52,356.76) .. controls (530.42,356.76) and (529.52,355.87) .. (529.52,354.76) -- cycle ;
\draw [color={rgb, 255:red, 208; green, 2; blue, 27 }  ,draw opacity=1 ]   (461.01,323.38) .. controls (455.19,332.11) and (497.82,351.77) .. (516.83,356.38) ;
\draw [shift={(518.52,356.76)}, rotate = 191.8] [color={rgb, 255:red, 208; green, 2; blue, 27 }  ,draw opacity=1 ][line width=0.75]    (10.93,-3.29) .. controls (6.95,-1.4) and (3.31,-0.3) .. (0,0) .. controls (3.31,0.3) and (6.95,1.4) .. (10.93,3.29)   ;
\draw  [color={rgb, 255:red, 208; green, 2; blue, 27 }  ,draw opacity=1 ][fill={rgb, 255:red, 208; green, 2; blue, 27 }  ,fill opacity=1 ] (527.52,325.76) .. controls (527.52,324.66) and (528.42,323.76) .. (529.52,323.76) .. controls (530.62,323.76) and (531.52,324.66) .. (531.52,325.76) .. controls (531.52,326.87) and (530.62,327.76) .. (529.52,327.76) .. controls (528.42,327.76) and (527.52,326.87) .. (527.52,325.76) -- cycle ;
\draw [color={rgb, 255:red, 65; green, 117; blue, 5 }  ,draw opacity=1 ][line width=1.5]    (362.51,156.51) -- (427.54,145.5) ;
\draw [shift={(430.5,145)}, rotate = 170.39] [color={rgb, 255:red, 65; green, 117; blue, 5 }  ,draw opacity=1 ][line width=1.5]    (14.21,-4.28) .. controls (9.04,-1.82) and (4.3,-0.39) .. (0,0) .. controls (4.3,0.39) and (9.04,1.82) .. (14.21,4.28)   ;
\draw [color={rgb, 255:red, 65; green, 117; blue, 5 }  ,draw opacity=1 ][line width=1.5]    (527,216) -- (527,268) ;
\draw [shift={(527,271)}, rotate = 270] [color={rgb, 255:red, 65; green, 117; blue, 5 }  ,draw opacity=1 ][line width=1.5]    (14.21,-4.28) .. controls (9.04,-1.82) and (4.3,-0.39) .. (0,0) .. controls (4.3,0.39) and (9.04,1.82) .. (14.21,4.28)   ;
\draw [color={rgb, 255:red, 208; green, 2; blue, 27 }  ,draw opacity=1 ]   (120.01,226.38) .. controls (114.19,235.11) and (136.12,233.13) .. (156.61,232.9) ;
\draw [shift={(158.51,232.88)}, rotate = 179.68] [color={rgb, 255:red, 208; green, 2; blue, 27 }  ,draw opacity=1 ][line width=0.75]    (10.93,-3.29) .. controls (6.95,-1.4) and (3.31,-0.3) .. (0,0) .. controls (3.31,0.3) and (6.95,1.4) .. (10.93,3.29)   ;
\draw [color={rgb, 255:red, 0; green, 0; blue, 124 }  ,draw opacity=1 ]   (209.5,137) .. controls (196.2,150.3) and (186.51,144.65) .. (176.58,138.69) ;
\draw [shift={(175.01,137.75)}, rotate = 30.76] [color={rgb, 255:red, 0; green, 0; blue, 124 }  ,draw opacity=1 ][line width=0.75]    (10.93,-3.29) .. controls (6.95,-1.4) and (3.31,-0.3) .. (0,0) .. controls (3.31,0.3) and (6.95,1.4) .. (10.93,3.29)   ;

\draw (195.51,202.91) node [anchor=north west][inner sep=0.75pt]  [color={rgb, 255:red, 0; green, 0; blue, 124 }  ,opacity=1 ]  {$\nu ^{*}$};
\draw (260.51,190.91) node [anchor=north west][inner sep=0.75pt]  [color={rgb, 255:red, 208; green, 2; blue, 27 }  ,opacity=1 ]  {$K$};
\draw (245.51,34.91) node [anchor=north west][inner sep=0.75pt]  [color={rgb, 255:red, 0; green, 0; blue, 124 }  ,opacity=1 ]  {$B\left( \nu ^{*} ,d\right)$};
\draw (109.51,213.91) node [anchor=north west][inner sep=0.75pt]  [font=\scriptsize,color={rgb, 255:red, 208; green, 2; blue, 27 }  ,opacity=1 ]  {$d/c$};
\draw (208.51,119.91) node [anchor=north west][inner sep=0.75pt]  [color={rgb, 255:red, 0; green, 0; blue, 124 }  ,opacity=1 ]  {$d$};
\draw (248.51,279.91) node [anchor=north west][inner sep=0.75pt]  [color={rgb, 255:red, 0; green, 0; blue, 0 }  ,opacity=1 ]  {$\mathrm{Y}$};
\draw (527.51,100.91) node [anchor=north west][inner sep=0.75pt]  [color={rgb, 255:red, 0; green, 0; blue, 124 }  ,opacity=1 ]  {$\nu ^{*}$};
\draw (585.51,77.91) node [anchor=north west][inner sep=0.75pt]  [color={rgb, 255:red, 208; green, 2; blue, 27 }  ,opacity=1 ]  {$K$};
\draw (510.51,9.91) node [anchor=north west][inner sep=0.75pt]  [color={rgb, 255:red, 0; green, 0; blue, 124 }  ,opacity=1 ]  {$B\left( \nu ^{*} ,d/2\right)$};
\draw (453.52,87.16) node [anchor=north west][inner sep=0.75pt]  [font=\scriptsize,color={rgb, 255:red, 208; green, 2; blue, 27 }  ,opacity=1 ]  {$d/2c$};
\draw (476.51,62.91) node [anchor=north west][inner sep=0.75pt]  [font=\scriptsize,color={rgb, 255:red, 0; green, 0; blue, 124 }  ,opacity=1 ]  {$d/2$};
\draw (577.51,162.91) node [anchor=north west][inner sep=0.75pt]  [color={rgb, 255:red, 0; green, 0; blue, 0 }  ,opacity=1 ]  {$\mathrm{Y}$};
\draw (531.51,334.91) node [anchor=north west][inner sep=0.75pt]  [color={rgb, 255:red, 0; green, 0; blue, 124 }  ,opacity=1 ]  {$\nu ^{*}$};
\draw (585.51,300.91) node [anchor=north west][inner sep=0.75pt]  [color={rgb, 255:red, 208; green, 2; blue, 27 }  ,opacity=1 ]  {$K$};
\draw (578.51,343.91) node [anchor=north west][inner sep=0.75pt]  [font=\scriptsize,color={rgb, 255:red, 0; green, 0; blue, 124 }  ,opacity=1 ]  {$B\left( \nu ^{*} ,d/4\right)$};
\draw (453.52,310.16) node [anchor=north west][inner sep=0.75pt]  [font=\scriptsize,color={rgb, 255:red, 208; green, 2; blue, 27 }  ,opacity=1 ]  {$d/4c$};
\draw (476.51,285.91) node [anchor=north west][inner sep=0.75pt]  [font=\scriptsize,color={rgb, 255:red, 0; green, 0; blue, 124 }  ,opacity=1 ]  {$d/4$};
\draw (577.51,385.91) node [anchor=north west][inner sep=0.75pt]  [color={rgb, 255:red, 0; green, 0; blue, 0 }  ,opacity=1 ]  {$\mathrm{Y}$};
\draw (178,407) node [anchor=north west][inner sep=0.75pt]  [font=\large,color={rgb, 255:red, 65; green, 117; blue, 5 }  ,opacity=1 ] [align=left] {Step 1};
\draw (363,118) node [anchor=north west][inner sep=0.75pt]  [font=\large,color={rgb, 255:red, 65; green, 117; blue, 5 }  ,opacity=1 ] [align=left] {Step 2};
\draw (536,228) node [anchor=north west][inner sep=0.75pt]  [font=\large,color={rgb, 255:red, 65; green, 117; blue, 5 }  ,opacity=1 ] [align=left] {Step 3};

\end{tikzpicture}
    \caption{Diagram of the first three iterations of Algorithm \ref{test}.}
    \label{fig:diagram}
\end{figure}
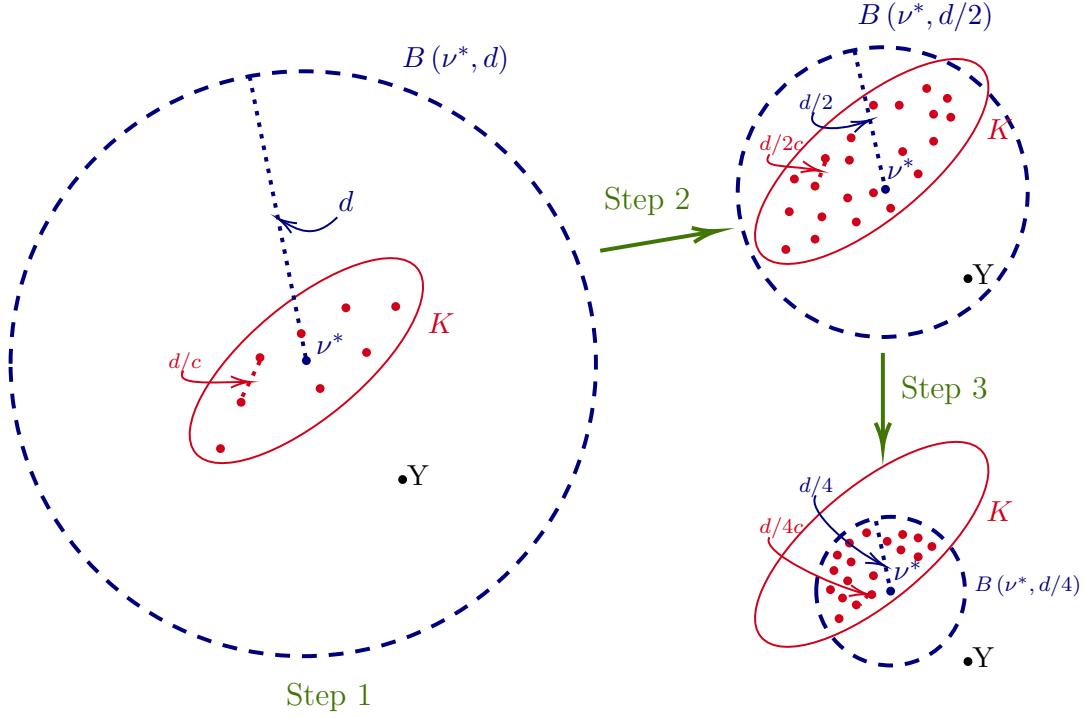



\begin{algorithm}
\SetKwComment{Comment}{/* }{ */}
\caption{Upper Bound Algorithm}\label{test}
\KwInput{A point $\nu^* \in K$}
$k \gets 1$\;
$\Upsilon \gets [\nu^*]$ \Comment*[r]{This array is needed solely in the proof and is not used by the estimator}
\While{TRUE}{
    Take a $\frac{d}{2^{k}(C+1)}$ maximal\footnotemark[1] packing set $M_k$ of the set $B\big(\nu^*, \frac{d}{2^{k-1}}\big) \cap K$ \Comment*[r]{The packing sets should be constructed prior to seeing the data}
    $\nu^* \gets \argmin_{\nu \in M_k} \|Y - \nu\| $ \Comment*[r]{Break ties by taking the point with the least lexicographic ordering}
    $\Upsilon$.append$(\nu^*)$\;
    $k \gets k + 1$\;
}
\Return{$\nu^*$}  \Comment*[r]{Observe that by definition $\Upsilon$ forms a Cauchy sequence\footnotemark[2], so $\nu^*$ can be understood as the limiting point of that sequence.}
\end{algorithm}
\footnotetext[1]{Here the maximality of the packing set is not really important; what is important is that the packing set is a covering. This can be ``constructed algorithmically'' by greedily taking points one by one and carving balls centered at those points.}

\footnotetext[2]{Take any two points $\Upsilon_{m}$ and  $\Upsilon_{m'}$ for $m' > m$. Then $\|\Upsilon_{m} - \Upsilon_{m'}\|\leq \sum_{i = m}^{m'-1} \|\Upsilon_i - \Upsilon_{i+1}\| \leq \sum_{i = m}^{m'-1} d/2^{i-1} \leq d/2^{m-2}$, so we have a Cauchy sequence.}

Before we proceed any further we will argue that the so defined estimator $\nu^* = \nu^*(Y)$ is a measurable function of the data. We have

\begin{theorem}\label{measurability:theorem}
The function $\nu^* : \RR^n \mapsto \RR^n$ is measurable (with respect to the Borel $\sigma$-field). As a consequence we have that $\nu^*(Y)$ is a random variable.
\end{theorem}

\begin{proof}
First we observe that for each $j$: $\Upsilon_j: \RR^n \mapsto \RR^n$ are measurable (here we denote by $\Upsilon_j$ the elements of the array $\Upsilon$ which is defined in Algorithm \ref{test}). In order to see this, we need to realize that one can (and should) construct the packing sets before one sees the data $Y$. This will form an infinite tree of packing sets rooted at the initial point $\Upsilon_1$. Each packing set splits $\mathbb{R}^n$ into polytopes (some of which may be unbounded) where each point in the packing set is the closest to any point in its corresponding polytope (this is the Voronoi tessellation in Euclidean norm). On the boundaries of these polytopes more than one point can be the closest point --- in that case in order to consistently assign a single point always take the point with the least lexicographic order (i.e. it has the smallest 1st coordinate of all points, and the smallest 2nd coordinate of all points with equally small first coordinate and so forth). 

Consider the event that $\Upsilon_j(y)$ belongs to a certain packing set, say, $M$ (i.e. the point $y$ is closest to all ancestor nodes of $M$ which essentially means that $y$ belongs to some intersection of polytopes (which is again a polytope call it $Q$)). For a point $m \in M$ we have that $\{y : \Upsilon_j(y) = m\} = (y \in P) \cap \{y : \Upsilon_j(y) \in M\} =  (y \in P) \cap (y \in Q) = (y \in P\cap Q)$, where $P$ is the polytope from the Voronoi tessalation given by $M$, of the point $m$. Since (convex) polytopes are comprised of finitely many linear inequalities they are Borel sets and hence the event $(\Upsilon_j(y) = m)$ is measurable. Repeating this argument for any point on the same width of the tree on which the point $m$ lies (i.e. on depth $j$ of the tree), shows that $\Upsilon_j$ is a measurable function and $\Upsilon_j(Y)$ is a discrete random variable. 

Next, we have $\nu^*(y) = \lim_{j} \Upsilon_j(y)$, where we know the limit exists since as we mentioned $\Upsilon_j(y)$ form a Cauchy sequence (hence a converging sequence) by definition. It suffices to check whether $\{y : \nu^*(y) \in B\}$ is a Borel set for any closed box $B$ (i.e., $B$ is a hyperrectangle parallel to the coordinate axes). Since 
\begin{align*}
    \{y : \nu^*(y) \in B\} = \bigcap_{j = 1}^n \{y : B_j^L \leq \nu^{j*}(y) \leq B_j^U\},
\end{align*}
where $\nu^{j*}$ denotes the $j$-th coordinate of $\nu^{*}$, and $ B_j^L$ and $ B_j^U$ are the upper and lower bounds of the box $B$ for the $j$-th coordinate, it suffices to show that the sets $\{y : B_j^L \leq \lim_i \Upsilon_i^{j}(y) \leq B_j^U\}$ are measurable. Note that since the sequence is converging 
\begin{align*}
    \lim_i \Upsilon_i^{j}(y) = \inf_{i \geq 1} \sup_{k \geq i} \Upsilon_k^{j}(y).
\end{align*}
Next
\begin{align*}
    \MoveEqLeft \{y : B_j^L \leq \lim_i \Upsilon_i^{j}(y) \leq B_j^U\} \\
    & = \{y: \inf_{i \geq 1} \sup_{k \geq i} \Upsilon_k^{j}(y) \leq B_j^U\}\bigcap\{y: B_j^L \leq \inf_{i \geq 1} \sup_{k \geq i} \Upsilon_k^{j}(y)\} \\
    & = \bigcap_{l \geq 1}\bigcup_{i \geq 1} \bigcap_{k \geq i} \{y : \Upsilon_k^{j}(y) \leq B_j^U + l^{-1}\} \bigcap \bigcap_{i \geq 1} \bigcup_{k \geq i} \{y : B_j^L \leq \Upsilon_k^{j}(y)\}.
\end{align*}
Finally note that the events $\{y : B_j^L \leq \Upsilon_k^{j}(y)\}$ and $\{y : \Upsilon_k^{j}(y) \leq B_j^U + l^{-1}\}$ are measurable since as we showed $\Upsilon_k$ are measurable, and the sets $\mathbb{R} \times \ldots (-\infty, B_j^U + l^{-1}] \times \mathbb{R}$ and $\mathbb{R} \times \ldots [B_j^L, \infty) \times \mathbb{R}$ are Borel sets in $\mathbb{R}^n$. This completes the proof.
\end{proof}

We will now argue that the estimator from Algorithm \ref{test} attains the minimax rate. The ideas we use are strongly inspired by the works of \cite{lecam1973convergence, birge1983approximation}. We start with a simple lemma.

\begin{lemma}\label{important:lemma} Suppose we are testing $H_0: \mu = \nu_1$ vs $H_A: \mu = \nu_2$ for $\|\nu_1 - \nu_2\| \geq C \delta$ for some $C > 2$. Then the test $\psi(Y) = \mathbbm{1}(\|Y - \nu_1\| \geq \|Y - \nu_2\|)$ satisfies
\begin{align*}
    \sup_{\mu: \|\mu- \nu_1\| \leq \delta}\mathbb{P}_{\mu}(\psi = 1) \vee \sup_{\mu: \|\mu- \nu_2\| \leq \delta}\mathbb{P}_{\mu}(\psi = 0) \leq \exp\bigg(-(C - 2)^2  \frac{\delta^2}{8 \sigma^2}\bigg).
\end{align*}
\end{lemma}

\begin{proof}Observe that
\begin{align*}
    \|Y - \nu_1\|^2 - \|Y - \nu_2\|^2 = 2(\mu + \xi)\T (\nu_2 - \nu_1) + \|\nu_1\|^2 - \|\nu_2\|^2.
\end{align*}
Suppose $\|\mu - \nu_1\| \leq \delta$. Then $\mu = \nu_1 + \eta$, $\|\eta\| \leq \delta$ and hence
\begin{align*}
    \MoveEqLeft 2(\mu + \xi)\T (\nu_2 - \nu_1) + \|\nu_1\|^2 - \|\nu_2\|^2 \\
    & = 2 \nu_1\T (\nu_2 - \nu_1) + 2\xi\T (\nu_2 - \nu_1) + \|\nu_1\|^2 - \|\nu_2\|^2 + 2\eta\T (\nu_2 - \nu_1) \\
    & = -\|\nu_1 - \nu_2\|^2 + 2 \eta\T (\nu_2 - \nu_1) + 2\xi\T (\nu_2 - \nu_1)
\end{align*}
We have $2 \eta\T (\nu_2 - \nu_1) \leq 2 \delta \|\nu_1 - \nu_2\| \leq \frac{2}{C} \|\nu_1 - \nu_2\|^2$. Hence the above is a normal with mean at most $(-1 + \frac{2}{C})\|\nu_1 - \nu_2\|^2 < 0$ (assuming $C > 2$) and variance equal to $4 \sigma^2\|\nu_1 - \nu_2\|^2$. By a standard bound on the normal distribution cdf \citep[see Section 2.2.1]{van1996weak} we have that 
\begin{align*}
    P(N(m, \tau^2) \geq 0) \leq \exp(-m^2/(2\tau^2)), 
\end{align*}
for $m < 0$, therefore the type I error of the test is bounded by 
\begin{align*}
    \exp\bigg(-\bigg(1 - \frac{2}{C}\bigg)^2 \frac{\|\nu_1 -\nu_2\|^2}{8 \sigma^2}\bigg) \leq \exp\bigg(-(C - 2)^2  \frac{\delta^2}{8 \sigma^2}\bigg).
\end{align*} By symmetry the same argument holds true for the type II error, namely when $\|\mu  - \nu_2\| \leq \delta$. 
\end{proof}

\begin{remark}
It is not too hard to see that this Lemma extends to centered sub-Gaussian noise. In other words if one supposes that $\xi$ satisfies $\EE \xi = 0$ and $\sup_{v \in S^{n-1}} \EE \exp(\lambda v\T \xi) \leq \exp(\overline \sigma^2\lambda^2/2)$ (where $S^{n-1}$ denotes the unit sphere in $\RR^n$) for some $\overline \sigma > 0$, the result becomes:
\begin{align*}
    \sup_{\mu: \|\mu- \nu_1\| \leq \delta}\mathbb{P}_{\mu}(\psi = 1) \vee \sup_{\mu: \|\mu- \nu_2\| \leq \delta}\mathbb{P}_{\mu}(\psi = 0) \leq \exp\bigg(-(C - 2)^2  \frac{\delta^2}{8 \overline \sigma^2}\bigg).
\end{align*}
Since Lemma \ref{important:lemma} is the only place which explicitly uses the Gaussian distribution (in the upper bound analysis), this automatically extends our upper bound results in the bounded $K$ case, for any centered sub-Gaussian noise with the change that $\sigma$ has to be substituted with the variance proxy $\overline \sigma$.
\end{remark}

Suppose now, we are given $M$ points $\nu_1,\ldots, \nu_M \in K' \subset K$ such that $\|\nu_i - \nu_j\| \geq \delta$ and $M$ is maximal\footnote[3]{We comment once again, that it is not the maximality that is important; rather it is important for the packing set to also be a covering set.}, i.e., we are given a maximal $\delta$-packing set of $K'$ and it is known that $\mu \in K' \subset K$. 

\begin{lemma}\label{most:importnant:lemma} Under the setting described above, let $i^* = \argmin_i \|Y - \nu_i\|$. We will show that the closest point to $Y$, $\nu_{i^*}$ satisfies
\begin{align*}
    \mathbb{P}(\|\nu_{i^*} - \mu\| > (C + 1)\delta) \leq M \exp(-(C - 2)^2 \delta^2/(8\sigma^2)),
\end{align*}
for any fixed $C > 2$.
\end{lemma}

\begin{proof}
Define the intermediate random variable
\begin{align*}
    T_i  = \begin{cases}
         \max_{j \in [M]}  \|\nu_i - \nu_j\|, \mbox{ s.t. } \|Y - \nu_i\| - \|Y - \nu_j\| \geq 0,\|\nu_i - \nu_j\| > C\delta\\
         0, \mbox{if no such $j$ exists},
    \end{cases}
\end{align*}
Without loss of generality assume that $\|\mu - \nu_i\| \leq \delta$ (here note that we have a $\delta$-packing which is also a $\delta$-covering). Next, we have that
\begin{align*}
    \mathbb{P}(\|\nu_{i^*} - \mu\| > \delta + C\delta) & \leq \mathbb{P}(i^* \in \{j: \|\nu_j - \nu_i\| > C\delta\})\\
    & \leq P(T_i > 0),
\end{align*}
where the first inequality follows by the triangle inequality and the second because if $i^* \in \{j: \|\nu_j - \nu_i\| \geq C\delta\}$ we have $T_{i} \geq \|\nu_i - \nu_{i^*}\| > C\delta$. But
\begin{align*}
    \mathbb{P}(T_i > 0) & = \mathbb{P}(\exists j: \|\nu_j - \nu_i\| > C\delta \mbox{ and } \|Y - \nu_i\| - \|Y - \nu_j\| \geq 0)
    \\ & \leq M \exp(-(C-2)^2 \delta^{2}/(8\sigma^2)),
\end{align*}
by Lemma \ref{important:lemma}. This is what we wanted to show.

\end{proof}
Finally we will need the following simple lemma.

\begin{lemma}\label{simple:lemma:monotone}The function $\varepsilon \mapsto M^{\operatorname{loc}}(\varepsilon)$ is monotone non-increasing.
\end{lemma}

\begin{remark}
This lemma heavily uses the fact that $K$ is a convex set.
\end{remark}

\begin{proof}
It suffices to show that the function $\varepsilon \mapsto M(\varepsilon/c, B(\theta, \varepsilon)\cap K)$ is non-increasing for any fixed $\theta \in K$. Upon rescaling one realizes that this is equivalent to packing the set $[\frac{1}{\varepsilon}(K - \theta)]\cap B(1)$ at a $1/c$ distance, where $B(1) = B(0,1)$ is the unit ball centered at $0$. Now we will show that if $\varepsilon' < \varepsilon$ we have $[\frac{1}{\varepsilon}(K - \theta)]\cap B(1)\subset [\frac{1}{\varepsilon'}(K - \theta)]\cap B(1)$. Clearly this is implied if we showed that $\frac{1}{\varepsilon}(K - \theta)\subset\frac{1}{\varepsilon'}(K - \theta)$. Take a point $ x \in \frac{1}{\varepsilon}(K - \theta)$.  Hence $ x = (k - \theta)/\varepsilon = 0 (\varepsilon-\varepsilon')/\varepsilon + \varepsilon'/\varepsilon (k - \theta)/\varepsilon'$ for some $k \in K$. Since $0, (k - \theta)/\varepsilon' \in \frac{1}{\varepsilon'}(K - \theta)$ and the set $\frac{1}{\varepsilon'}(K - \theta)$ is convex, this completes the proof.
\end{proof}

Finally we are in a good position to show the main result regarding the estimator of Algorithm \ref{test}.

\begin{theorem} \label{upper:bound:rate}The estimator from Algorithm \ref{test} returns a vector $\nu^*$ which satisfies the following property
\begin{align*}
    \mathbb{E} \|\mu - \nu^*\|^2 \leq \bar C \varepsilon^{*2},
\end{align*}
for some universal constant $\bar C$. Here $\varepsilon^* = \varepsilon_{J^*}$ and $J^*$ is the maximal $J \geq 1$, $J \in \mathbb{N}$, such that $\varepsilon_J := \frac{d (c/2 - 3)}{2^{J -2}c}$ satisfies
\begin{align}\label{upper:bound:suff:cond}
    \frac{\varepsilon_J^2}{\sigma^2} > 16\log M^{\operatorname{loc}}\bigg(\varepsilon_J \frac{c}{(c/2-3)}\bigg) \vee 16\log 2,
\end{align}
or $J^* = 1$ if no such $J$ exists. We remind the reader that $c$ is the constant from the definition of local entropy, which is assumed to be sufficiently large.
\end{theorem}
\begin{proof}
Combining the results of Lemma \ref{most:importnant:lemma} (with $c = 2(C+1)$ where $c$ is the constant from the definition of local packing entropy) and Lemma \ref{simple:lemma:monotone} we can conclude that for any $2 \leq j \leq J$
\begin{align*}
    \mathbb{P}\bigg(\|\mu - \Upsilon_{j}\| > \frac{d}{2^{j-1}} \bigg| \|\mu - \Upsilon_{j-1}\| \leq \frac{d}{2^{j-2}}, \Upsilon_{j-1}\bigg) & \leq |M_{j-1}| \exp\bigg(-\frac{(C-2)^2d^2}{(2^{2(j-1)}(C+1)^2)8\sigma^2}\bigg)\\
    & \leq  M^{\operatorname{loc}}\bigg(\frac{d}{2^{J - 2}}\bigg)\exp\bigg(-\frac{(C-2)^2d^2}{(2^{2(j-1)}(C+1)^2)8\sigma^2}\bigg).
\end{align*}
where $M_{j-1}$ is the packing sets from Algorithm \ref{test} corresponding to $\Upsilon_{j-1}$. Since the bound does not depend on $\Upsilon_{j-1}$ we can drop it from the conditioning. Telescoping this bound (i.e., using that for $k$ events $\{A_i\}_{i \in [k]}$ such that $\PP(A_i^c) > 0, i \in [k-1]$, it always holds that $\PP(A_k) \leq \PP(A_k | A_{k-1}^c) + \PP(A_{k-1} | A_{k-2}^c) + \ldots + \PP(A_2 | A_1^c) + \PP(A_1)$, which can be proved by induction) we obtain

\begin{align}\label{important:telescoping:bound:which:is:true}
    \mathbb{P}(\|\mu - \Upsilon_{J}\| > \frac{d}{2^{J-1}}) &
    \leq M^{\operatorname{loc}}\bigg(\frac{d}{2^{J - 2}}\bigg) \sum_{j = 1}^{J-1}\exp\bigg(-\frac{(C-2)^2d^2}{(2^{2j}(C+1)^2)8\sigma^2}\bigg)\nonumber\\
    & \leq M^{\operatorname{loc}}\bigg(\frac{d}{2^{J - 2}}\bigg) a (1 + a^{4-1} + a^{16-1} + \ldots)\mathbbm{1}(J > 1)\nonumber\\
    & \leq M^{\operatorname{loc}}\bigg(\frac{d}{2^{J - 2}}\bigg) \frac{a}{1-a} \mathbbm{1}(J > 1),
\end{align}
where for brevity we put \begin{align*}a =  \exp\bigg(\frac{-(C-2)^2d^2}{(2^{2(J - 1)}(C+1)^2)(8\sigma^2)}\bigg),
\end{align*}
and we are assuming that $ a < 1$. So if one sets $\varepsilon_J = \frac{(C-2)d}{2^{J - 1}(C+1)}$, we have that if $\varepsilon_J^2/(8\sigma^2) > 2 \log M^{\operatorname{loc}}\bigg(\varepsilon_J \frac{2 (C+1)}{(C-2)}\bigg)$ and $a = \exp(-\varepsilon_J^2/(8\sigma^2)) < 1/2$, the above probability will be bounded from above by $2\exp(-\varepsilon_J^2/(16\sigma^2))$. Since $2 \log M^{\operatorname{loc}}\bigg(\varepsilon_J \frac{2 (C+1)}{(C-2)}\bigg) < 2 \bigg(\log 2 \vee \log M^{\operatorname{loc}}\bigg(\varepsilon_J \frac{2 (C+1)}{(C-2)}\bigg)\bigg)$ this condition is implied when 
\begin{align}\label{suff:condition:epsJ}
    \frac{\varepsilon_J^2}{\sigma^2} > 16\log M^{\operatorname{loc}}\bigg(\varepsilon_J \frac{2 (C+1)}{(C-2)}\bigg) \vee 16\log 2.
\end{align} 

By the triangle inequality we have that 
\begin{align}\label{mu:upislon:ineq}
    \|\nu^* - \mu\| \leq \|\nu^* - \Upsilon_J\| + \|\Upsilon_J - \mu\| \leq 3\varepsilon_J\frac{C+1}{C-2},
\end{align} with probability at least $1 - 2\exp(-\varepsilon_J^2/(16\sigma^2))$ which holds for all $J$ satisfying \eqref{suff:condition:epsJ}. Here we want to clarify that the last inequality in \eqref{mu:upislon:ineq} follows from the fact that $\|\nu^* - \Upsilon_J\| \leq d/2^{J-2}$, as seen when we verified that $\Upsilon$ forms a Cauchy sequence.
Let $J^*$ be selected as the maximum $J$ such that \eqref{suff:condition:epsJ} holds, or otherwise if such $J$ does not exist $J^* = 1$. Let $\kappa = 3 \frac{C+1}{C-2}$, $\underline{C} = 2$ and $C' = \frac{1}{16}$. We have established that the following bound holds:
\begin{align*}
    \mathbb{P}(\|\mu - \nu^*\| > \kappa \varepsilon_J) \leq \underline C \exp(-C'\varepsilon_J^2/\sigma^2) \mathbbm{1}(J > 1) \leq \underline C \exp(-C'\varepsilon_J^2/\sigma^2) \mathbbm{1}(J^* > 1),
\end{align*}
for all $1 \leq J \leq J^*$, where this bound also holds in the case when $J^* = 1$ by exception. Observe that we can extend this bound to all $J \in \mathbb{Z}$ and $J \leq J^*$, since for $J < 1$ we have $\kappa \varepsilon_J \geq 6 d$ and so 
\begin{align*}
    \mathbb{P}(\|\mu - \nu^*\| > \kappa \varepsilon_J) \leq 0 \leq \underline C \exp(-C'\varepsilon_J^2/\sigma^2) \mathbbm{1}(J^* > 1).
\end{align*}
Now for any $\varepsilon_{J-1} > x \geq \varepsilon_{J}$ for $J \leq J^*$ we have that 
\begin{align*}
    \mathbb{P}(\|\mu - \nu^*\| > 2 \kappa x) & \leq \mathbb{P}(\|\mu - \nu^*\| \geq \kappa \varepsilon_{J-1}) \leq \underline C \exp(-C'\varepsilon_{J-1}^2/\sigma^2) \mathbbm{1}(J^* > 1)\\
    & \leq \underline C \exp(-C'x^2/\sigma^2)\mathbbm{1}(J^* > 1),
\end{align*}
where the last inequality follows due to the fact that the map $x \mapsto \underline C \exp(-C'x^2/\sigma^2)$ is monotonically decreasing for positive reals. We will now integrate the tail bound:
\begin{align}\label{important:prob:bound}
 \mathbb{P}(\|\mu - \nu^*\| \geq 3 \kappa x) \leq \mathbb{P}(\|\mu - \nu^*\| > 2 \kappa x) \leq \underline C \exp(-C'x^2/\sigma^2) \mathbbm{1}(J^* > 1),
\end{align}
which holds true for $x \geq \varepsilon^*$ (for $\varepsilon^* > 0$; if $\varepsilon^* = 0$, that means $\sigma = 0$ in which case we know the algorithm outputs the correct point), where $\varepsilon^* = \varepsilon_{J^*} = \frac{(C-2)d}{(C+1) 2^{J^* - 1}}$, always (since even if $J^* = 1$ by exception, this bound is still valid).

We have
\begin{align*}
    \mathbb{E} \|\mu - \nu^*\|^2 & = \int_{0}^{\infty} 2 x \mathbb{P}(\|\mu - \nu^*\| \geq x) dx \\
    & \leq C''' \varepsilon^{*2} + \int_{3\kappa\varepsilon^*}^{\infty} 2 x \underline C \exp(-C''x^2/\sigma^2) \mathbbm{1}(J^* > 1) dx \\
    & = C''' \varepsilon^{*2} + C^{''''} \sigma^2\exp(-C'''''\varepsilon^{*2}/\sigma^2)\mathbbm{1}(J^* > 1).
\end{align*}
Now $\varepsilon^{*2}/\sigma^2$ is bigger than a constant ($16 \log 2$) otherwise $J^* = 1$. Hence the above is smaller than $\bar C \varepsilon^{*2}$ for some absolute constant $\bar C$.
\end{proof}

We will now formally illustrate that the above estimator achieves the minimax rate. The precise expression of the rate is quantified in the following result:


\begin{theorem}\label{minimax:rate:thm} Define $\varepsilon^*$ as $\sup \{\varepsilon: \varepsilon^2/\sigma^2 \leq \log M^{\operatorname{loc}}(\varepsilon)\}$, where $c$ in the definition of local entropy is a  sufficiently large absolute constant. Then the minimax rate is given by $\varepsilon^{*2} \wedge d^2$ up to absolute constant factors.
\end{theorem}

\begin{proof}
First suppose that $\varepsilon^*$ satisfies $\varepsilon^{*2}/\sigma^2 > 16\log 2$. Then for $\delta^* := \varepsilon^*/4$ we have $\log M^{\operatorname{loc}}(\delta^*) \geq\log M^{\operatorname{loc}}(\varepsilon^*) \geq  \varepsilon^{*2}/(2\sigma^2) + \varepsilon^{*2}/(2\sigma^2) > 8\delta^{*2}/\sigma^2 + 8\log 2$ and so this implies the sufficient condition for the lower bound. 

On the other hand we know that for a constant $C > 1$: 
\begin{align*}
    4 C \varepsilon^{*2}/\sigma^2 \geq C \log M^{\operatorname{loc}}(2\varepsilon^*) \geq C \log M^{\operatorname{loc}}(2\varepsilon^*\sqrt{C}) \geq C \log M^{\operatorname{loc}}\bigg(2\varepsilon^*\sqrt{C}\frac{c}{c/2-3}\bigg),  
\end{align*}
and so setting $\delta = 2 \varepsilon^* \sqrt{C}$ we obtain that 
\begin{align*}
    \delta^2/\sigma^2 \geq C \log M^{\operatorname{loc}}\bigg(\delta\frac{c}{c/2-3}\bigg).
\end{align*}
For $C = 16$ this will satisfy the inequality \eqref{upper:bound:suff:cond} (taking into account that $\varepsilon^{*2}/\sigma^2 > 16\log 2$, which implies $\delta^2/\sigma^2 \geq 64 \log 2 C > 16 \log 2$). Since the map $x \mapsto x^2/\sigma^2 - 16\log M^{\operatorname{loc}}\bigg(x\frac{c}{c/2-3}\bigg) \vee 16\log 2$ is non-decreasing, we have that $\delta \geq \varepsilon_{J^*}/2$. This shows that the rate in this case is $\varepsilon^{*2}$.

Next, suppose that $\varepsilon^{*}$ defined by $\sup \{\varepsilon: \varepsilon^2/\sigma^2 \leq \log M^{\operatorname{loc}}(\varepsilon)\}$ satisfies $\varepsilon^{*2}/\sigma^2 \leq 16\log 2$. For $2\varepsilon^*$, we have $64 \log 2 > 4\varepsilon^{*2}/\sigma^2 \geq \log M^{\operatorname{loc}}(2 \varepsilon^*)$. If $c$ in the definition of local packing is large enough, we could put points in the diameter of the ball with radius $2 \varepsilon^*$ such that the packing set has more than $\exp(64\log 2)$ many points. But that implies that the set $K$ is entirely inside a ball of radius $\sqrt{(64\log 2)} \sigma$ (as $\varepsilon^{*2} \leq 16\log 2 \sigma^2 $). In such a case, for the lower bound, we could pick $\varepsilon$ to be proportional to the diameter of the set (with a small proportionality constant). That will ensure that $\varepsilon/\sigma$ is upper bounded by some constant (as $2\sqrt{(64 \log 2)}\sigma$ is bigger than the diameter), and at the same time $\log M^{\operatorname{loc}}(\varepsilon)$ can be made bigger than a constant (provided that $c$ in the definition of a local packing is large enough) -- by taking $\theta$ (where $\theta$ is the center of the localized set $B(\theta, \varepsilon) \cap K$) to be the midpoint of a diameter of the set $K$ and then placing equispaced points on the diameter. Hence the diameter of the set is a lower bound (up to constant factors) in this case, which is of course always an upper bound too (up to constant factors). So we conclude that either for $\varepsilon^{*}$ defined by $\sup \{\varepsilon: \varepsilon^2/\sigma^2 \leq \log M^{\operatorname{loc}}(\varepsilon)\}$ satisfies $\varepsilon^{*2}/\sigma^2 > 16\log 2$ or the lower and upper bounds are of the order of the diameter of the set. In summary the rate is given by the $\varepsilon^{*2} \wedge d^2$. This is true since in the second case, $4\varepsilon^*$ is bigger than the diameter of the set.
\end{proof}

In practice it may be challenging to calculate $\varepsilon^*$ precisely, but the following lemma can be useful. 

\begin{lemma} Suppose that $\varepsilon$ and $\varepsilon'$ are such that $\varepsilon^2/\sigma^2 > \log M^{\operatorname{loc}}(\varepsilon)$ and $\varepsilon^{'2}/\sigma^2 < \log M^{\operatorname{loc}}(\varepsilon')$ and $\varepsilon \asymp \varepsilon'$. Then the rate is given by $\varepsilon^2 \wedge d^2$.
\end{lemma}

\begin{proof} It is clear from the definition of $\varepsilon^*$ that $\varepsilon \geq \varepsilon^*$ while $\varepsilon' \leq \varepsilon^*$. Since $\varepsilon \asymp \varepsilon'$ it follows that $\varepsilon \asymp \varepsilon^*$ which grants the result.
\end{proof}


\begin{remark}
It should be clear that $M^{\operatorname{loc}}(\varepsilon)$ can be bounded using Sudakov minoration to yield an upper bound on the minimax rate. We give details in this remark as follows. Suppose that $\frac{\varepsilon^2}{\sigma^2} \geq 4 c^{-2} \log M^{\operatorname{loc}}(\varepsilon)$. Clearly upon rescaling such an $\varepsilon$ (by $c/2$) we can obtain $\varepsilon' = \frac{\varepsilon c}{2}$ (which is of the same order) and is $\geq \varepsilon^*$. The latter follows by the fact that $\frac{(\varepsilon c)^2}{4\sigma^2} \geq \log M^{\operatorname{loc}}(\varepsilon) \geq \log M^{\operatorname{loc}}(\frac{\varepsilon c}{2})$ since $c$ is sufficiently large. By Sudakov minoration we have $\log M^{\operatorname{loc}}(\varepsilon) \leq sup_{\theta \in K} \frac{w(B(\theta,\varepsilon) \cap K)^2}{\varepsilon^2/c^2}$, where $w$ denotes the Gaussian width \citep[see Section 5]{wainwright2019high}. It follows that if there exists an $\varepsilon$ such that $\frac{\varepsilon^2}{2\sigma} \geq sup_{\theta \in K} w(B(\theta,\varepsilon) \cap K)$ the minimax rate is upper bounded by $\varepsilon^2 \wedge d^2$. An alternative way of seeing that this upper bound on the minimax rate holds, is to use Theorem 2.3. of \cite{bellec2018sharp}, which shows that the constrained LSE grants this rate. We will also see in our examples, that there exists another universal upper bound on the minimax rate in terms of Kolmogorov complexity. An alternative way of seeing that bound, will be to use the projection estimator $P Y$ where $P$ is an orthogonal projection selected in a certain way (cf. Section \ref{upper:bound:example:section} for more details). 
\end{remark}

\section{Examples} \label{examples:section}

We now consider several examples, which have been studied previously; nevertheless we find it enlightening to study them from this new perspective. Our examples are also meant to show the reader a couple of methods one can utilize to attain bounds on the local entropy of the constraint set. In addition we will consider an example of convex weak $\ell_p$ balls, and an example of bounded polytopes with $N$ vertices, both of which have not been previously studied to the best of our knowledge. The first example we consider below is concerned with hyperrectangles.

\subsection{Hyperrectangles}

Let $K = \prod_{i = 1}^n \bigg[-\frac{a_i}{2}, \frac{a_i}{2}\bigg] \subset \mathbb{R}^n$ be a hyperrectangle. Without loss of generality we will assume that $0 < a_1 \leq a_2 \leq \ldots \leq a_n$. We will show that the following result holds:
\begin{corollary}
The rate when $K$ is a hyperrectangle as above is given by $(k + 2) \sigma^2 \wedge d^2$ (for $d^2 = \sum_{i \in [n]} a_i^2$) where $k \in \{0,\ldots,n-1\}$ is such that $(k + 1) \sigma^2 \leq \sum_{i = 1}^{n-k} a_i^2$ but $(k + 2)\sigma^2 > \sum_{i = 1}^{n-k-1} a_i^2$, and in the case when $\sum_{i = 1}^n a_i^2 \leq \sigma^2$ the rate is $d^2$.
\end{corollary}

\subsubsection{Upper Bound}

For the upper bound it suffices to consider the case when $\sum_{i = 1}^n a_i^2 > \sigma^2$ (otherwise the rate is $d^2$ which can trivially be achieved).

Suppose we select $\varepsilon > c' \sqrt{k + 2} \sigma$, for $c'$ being a large constant. We need to make an $\varepsilon/c$ packing of the set $B(\theta,\varepsilon) \cap K$ for any $\theta \in K$. Suppose $M_\theta$ is the corresponding packing set. Take any two points $x,y\in M_\theta$. We have
\begin{align*}
    \varepsilon/c & \leq \|x-y\| \leq \|x_1^{n-k-1} - y_1^{n-k-1}\| + \|x_{n-k}^{n} - y_{n-k}^{n}\|\\
    & \leq \sqrt{\sum_{i = 1}^{n-k-1} a_i^2} + \|x_{n-k}^{n} - y_{n-k}^{n}\|\\
    & \leq \sqrt{k+2}\sigma + \|x_{n-k}^{n} - y_{n-k}^{n}\|,
\end{align*}
where we denoted by $x_{l}^m = (x_l, x_{l + 1}, \ldots, x_m)\T$. Hence for a large enough $c'$ we will have 
\begin{align*}
     \|x_{n-k}^{n} - y_{n-k}^{n}\| \geq \varepsilon/c'',
\end{align*}
where $c'' = (c'/c - 1)$. This means, that the packing set, also forms a $\varepsilon/c''$ packing on the last $k + 1$ coordinates. However, this set can at most be a $(k + 1)$-sphere with radius $\varepsilon$, and so such a packing number will be bounded by $(k + 1)\log (1 + 2c'') \ll (c' \sqrt{k + 2})^2$ \citep{wainwright2019high} for a large $c'$.

\subsubsection{Lower Bound}

Next for the lower bound, we will show a lemma first.
\begin{lemma}The log cardinality of a maximal packing set of a $k$-dimensional hypercube with side length $\sigma$, to a distance $\sqrt{k}\sigma/c$ for some sufficiently large $c$, is at least $\bar c k$ for some $\bar c > 0$.
\end{lemma}
\begin{proof} For $k = 1$ the assertion is obviously true, so we assume $k \geq 2$. We know that the packing number is at least the ratio between the volumes \citep{wainwright2019high}. The volume of the hypercube is $\sigma^k$. The volume of a sphere of radius $\sqrt{k}\sigma/c$ is $\frac{(\sqrt{k}\sigma/c)^k \pi^{k/2}}{\Gamma(k/2 + 1)}$. Taking the ratio we obtain
\begin{align*}
    \frac{c^k \Gamma(k/2 + 1)}{\sqrt{k}^k \pi^{k/2}}.
\end{align*}
If $k$ is even, by Stirling's approximation 
\begin{align*}
    \Gamma(k/2 + 1) = (k/2)! > \sqrt{2\pi} (k/2)^{k/2 + 1/2} \exp(-k/2)\exp(1/(6k + 1)).
\end{align*} For $c$ large enough, the log of the ratio can then be lower bounded by $k \log[ c/(\sqrt{2\pi}\exp(1/2))] + \frac{1}{2}\log(k/2) + \log(\sqrt{2\pi})- \frac{1}{6k + 1}$. On the other hand, for odd $k$, since $\Gamma$ is increasing (on the interval $[2,\infty)$), we have $\Gamma(k/2 + 1) \geq \Gamma((k-1)/2 + 1) > \sqrt{2\pi} ((k-1)/2)^{(k-1)/2 + 1/2} \exp(-(k-1)/2)\exp(1/(6(k-1) + 1))$, so that the same conclusion holds.
\end{proof}
Going back to the lower bound let us first suppose that $d^2 > \sigma^2$. We will now construct a $\lceil (k + 1)/2\rceil$-dimensional hyperrectangle with side length at least $\sigma$ out of the given points. First, assume that $s$ of the $a_i^2$ are at least $\sigma^2$. If $s \geq k$ then we can build a $k$-dimensional hyperrectangle of side lengths at least $\sigma$. In case $s < k$, we know all of the remaining $n-s$ coordinates are $< \sigma$. Hence by greedily taking coordinates until we reach $\sigma^2$ (and note that any such summation will be smaller than $2\sigma^2$) we can construct a hyperrectangle of dimension at least $\lceil (k + 1)/2\rceil$ with sides at least $\sigma$ (here we are using the fact that $(k + 1) \sigma^2 \leq \sum_{i = 1}^{n-k} a_i^2$ by assumption). If we build a sphere centered at the center of this hyperrectangle of radius $\sqrt{\lceil (k + 1)/2\rceil}\sigma$, this sphere contains a hypercube of side $\sigma$, which is fully inside the hyperrectangle. When $c$ from the definition of local packing is sufficiently large, this hypercube can be packed with at least $\exp(\bar c \lceil (k + 1)/2\rceil)$ points according to the lemma above. Hence for $\varepsilon' = \sqrt{\lceil (k + 1)/2\rceil}\sigma$ we have $\varepsilon^{'2}/\sigma^2 \lesssim \log M^{\operatorname{loc}}(\varepsilon')$. Thus by rescaling $\varepsilon'$ we can obtain $\varepsilon^{'2}/\sigma^2 <  \log M^{\operatorname{loc}}(\varepsilon')$. Hence the conclusion.

The last case is to consider $d^2 < \sigma^2$. This case can be handled by the same logic, as in the proof of Theorem \ref{minimax:rate:thm} since $d < \sigma$. This completes the proof. 

\subsection{Ellipses}

Next we consider the example of ellipses. Let $K = \{x : \sum_i \frac{x_i^2}{a_i} \leq 1\}$, where we assume $0 < a_1 \leq \ldots \leq a_n$. Define the Kolmogorov width \citep{pinkus2012n} as
\begin{align}\label{K:width}
    d_k(K) = \min_{P \in \mathcal{P}_{k}}\max_{\theta \in K} \|P \theta - \theta \|,
\end{align}
where $\mathcal{P}_k$ denotes the set of all $k$-dimensional linear projections. It is known that that $d_k(K) = \sqrt{a_{n - k}}$, where $a_0 = 0$ \cite[see, e.g.,][and references therein]{wei2020gauss}. Below we will show the following result:
\begin{corollary}
The minimax rate for ellipses is $(k + 1)\sigma^2 \wedge d^2$, where $k \in [n]$ is such that ${a_{n-k}} \leq (k + 1)\sigma^2$ but ${a_{n-k + 1}} > k\sigma^2$, or $d^2$ in the case ${a_n} \leq \sigma^2$.
\end{corollary}

\subsubsection{Upper Bound}

The upper bound proof is very similar to the bound for the hyperrectangles. We will only focus on the case $a_n > \sigma^2$ as otherwise the upper bound is trivial. Suppose $\varepsilon^2 > C k \sigma^2$. We need an $\varepsilon/c$ packing set. Take two points $x,y$ in that packing set and let $P$ be the projection achieving the $\min$ in \eqref{K:width}. We have
\begin{align*}
    \varepsilon/c \leq \|x- y\| \leq \|x -  Px - y + Py\| + \|Px - Py\|\leq 2 d_k(K) + \|Px - Py\|
\end{align*}
But $d^2_k(K) \leq (k + 1) \sigma^2$ so when $C$ is sufficiently large we have
\begin{align*}
     \|Px - Py\| \geq \varepsilon/c''.
\end{align*}
But this is a $k$-dimensional set, which is at most a $k$-sphere, which means that the packing set is of cardinality at most $k C''$. Hence by potentially rescaling $\varepsilon$ to some bigger value, we will obtain $\varepsilon^2/\sigma^2 > \log M^{\operatorname{loc}}(\varepsilon)$.

\subsubsection{Lower Bound}

For the lower bound, observe that the ellipse, contains a $k$-dimensional ball of radius $\sqrt{k\sigma^2}$. This can be seen by setting the first $n-k$ coefficients to $0$ and then having the set 
\begin{align*}
    \sum_{i \geq n- k + 1} \frac{x_i^2}{a_i} \leq 1,
\end{align*}
and since $a_{n - k + 1} \geq k\sigma^2$ we have the ball inside. This ball can be packed with at least $kC$ log-packing. Hence the lower bound upon rescaling $\varepsilon^2 = k\sigma^2$ down a bit. 

The only case that we have not handled is if $a_i \leq \sigma^2$ for all $i$ (which implies that the diameter is also smaller than $\sigma$). But that can be handled as in Theorem \ref{minimax:rate:thm} to yield a rate equal to the diameter of the set.

It is worth pointing out here that the LSE fails to be minimax optimal for certain ellipses. This is shown in \cite{zhang2013nearly} for instance, see their Lemma 7. For a different example of when the LSE fails refer to \cite{chatterjee2014new}.

\subsection{Compact Orthosymmetric Quadratically Convex Sets}

In this section we consider an example of sets which was first proposed and analyzed in \cite{donoho1990minimax}. The compact convex set $K$ is called orthosymmetric if for $x = (x_1, \ldots, x_n)\T \in K$ we have $(\pm x_1, \ldots, \pm x_n)\T \in K$ for all possible choices of $\pm$. The set is called quadratically convex if $K^2 := \{x^2 : x \in K\}$ is a convex set, where $x^2$ is $x$ squared entry-wise. Examples of such sets are hyperrectangles and ellipses. For even more examples refer to \cite{donoho1990minimax}. We have

\begin{corollary}
Using the definition of Kolmogorov widths the minimax rate is given by $(k + 1)\sigma^2 \wedge d_0(K)^2$ where $k$ is such that $d_k(K)^2 \leq (k + 1)\sigma^2$ but $d^2_{k-1}(K) > k \sigma^2$. If $d_0(K)^2 \leq \sigma^2$ we have that the rate is $d_0(K)^2$ which is up to constants the diameter of the set. 
\end{corollary}

\subsubsection{Upper Bound} \label{upper:bound:example:section}

The upper bound is the same as in the ellipse case, and in fact this upper bound is always valid. This reflects the fact that one can always use the optimal projection $P Y$ to estimate $\mu$.

\subsubsection{Lower Bound}

For the lower bound we may assume
\begin{align*}
    \min_{P \in \mathcal{P}_{k-1}} \max_{\theta \in K} \|\theta - P\theta\|^2 \geq k \sigma^2.
\end{align*}
We can only consider projections aligned with the coordinates -- there are $n_k := {n \choose k-1}$ such projections. Then the optimization is
\begin{align*}
    \min_{P \in \mathcal{P}_{k-1}} \max_{\theta \in K} \|\theta - P\theta\|^2 \leq \min_{S} \max_{\theta \in K} \sum_i \theta^2_i - \sum_{i \in S}\theta_i^2,
\end{align*}
where the minimum over $S$ is taken with respect to all subsets of $[n]$ with exactly $k-1$ elements. Since the set is quadratically convex the above can be written as
\begin{align*}
    \min_{P \in \mathcal{P}_{k-1}} \max_\theta \|\theta - P\theta\|^2 \leq \min_S \max_{t \in K^2} \sum_i t_i - \sum_{i \in S}t_i = \min_{w \in S_k} \max_{t \in K^2} \mathbbm{1}\T t - w\T t,
\end{align*}
where $w$ ranges in the set $S_k := \{e: e \in \RR^n$ has exactly $k-1$-entries equal to $1$ and the rest are $0\}$ and $\mathbbm{1} \in \RR^n$ denotes the vector comprised of $1$'s. It follows that for each $w_i \in S_k$ there exists a $t_i$ such that $\mathbbm{1}\T t_i - w_i\T t_i \geq k \sigma^2$. Since the set $K$ is convex and orthosymmetric we may assume without loss of generality that $t_i$ has $0$ entries on the support of $w_i$ and $\mathbbm{1}\T t_i = \mathbbm{1}\T t_i - w_i\T t_i = k \sigma^2$ (the latter holds since the set contains $0$). We will now argue that there exists a convex combination $t_{\alpha} := \sum_{i \in [n_k]} \alpha_i t_i$ such that $\|t_{\alpha}\|_{\infty} \leq \sigma^2$, where, as usual, $\|t_{\alpha}\|_{\infty}$ denotes the maximum of the absolute values of the entries of the vector $t_{\alpha}$. To see this, first observe that since all $t_i$ have positive entries $\|t_{\alpha}\|_{\infty} = \max_{e: e \geq 0, \mathbbm{1}\T e = 1} e\T t_{\alpha}$. Hence it suffices to show that 
\begin{align*}
    \min_{\alpha: \alpha \geq 0, \mathbbm{1}\T \alpha = 1} \max_{e: e \geq 0, \mathbbm{1}\T e = 1} e\T t_{\alpha} \leq \sigma^2.
\end{align*}
Since both sets over which the optimization is performed are convex, and the function $ e\T t_{\alpha} = e\T \sum_{i \in [n_k]} \alpha_i t_i$ is convex-concave (indeed it is linear in both arguments) by the minimax theorem we have
\begin{align*}
    \min_{\alpha: \alpha \geq 0, \mathbbm{1}\T \alpha = 1} \max_{e: e \geq 0, \mathbbm{1}\T e = 1} e\T t_{\alpha} = \max_{e: e \geq 0, \mathbbm{1}\T e = 1} \min_{\alpha: \alpha \geq 0, \mathbbm{1}\T \alpha = 1} e\T t_{\alpha} = \max_{e: e \geq 0, \mathbbm{1}\T e = 1} \min_{i \in [n_k]} e\T t_{i}.
\end{align*}
Observe that $\min_{i \in [n_k]} e\T t_{i} \leq e\T t_{e}$, where $t_e$ is selected such that it has $0$ entries corresponding to the top $k-1$ entries of $e$. Thus $e\T t_e = \sum_{i = 1}^{n - k + 1} e_{(i)} t_{e,(i)} \leq e_{(n - k + 1)} \sum t_{e, (i)} = e_{(n - k + 1)} k \sigma^2$, where $e_{(i)}$ denote the order statistics for the entries of the vector $e$, i.e. $e_{(1)} \leq e_{(2)} \leq \ldots \leq e_{(n - k + 1)}$ and $t_{e,(i)}$ are the concomitant values from the entries of $t_e$. Finally observe that $e_{(n - k + 1)} \leq \frac{1 - \sum_{i = 1}^{n-k} e_{(i)}}{k} \leq \frac{1}{k}$. Hence we conclude that there exists $t^* = t_{\alpha^*}$ such that
\begin{align*}
    \|t^*\|_{\infty} \leq \sigma^2.
\end{align*}
In addition since $t^*$ is a convex combination of vectors $t_i$ we must have $\mathbbm{1}\T t^* = k\sigma^2$, and $t^* \in K^2$.


Since the set is orthosymmetric we have the hyperrectangle $\prod_{i \in [n]} [-\sqrt{t_i^*}, \sqrt{t_i^*}] \subset K$. Hence the logic is the same as in the hyperrectangular case --- we know that all entries of $t^*$ are smaller than $\sigma^2$ and they sum up to $k\sigma^2$. Hence we can create a large ($\lceil k/2 \rceil$-dimensional) hyperrectangle of side lengths at least $\sigma$, and the proof can continue as in the hyperrectangle case. The final case to consider is when $d_0(K)^2 \leq \sigma^2$, but that can be handled as in Theorem \ref{minimax:rate:thm}.

\subsection{$\ell_1$ ball}\label{yang:barron:section:exmaple}

In this section we will replicate a result of \cite{donoho1994minimax}. Suppose the set $K = \{\theta: \|\theta\|_1 \leq 1\}$. We will use the fact that
\begin{align}\label{yang:barron:bound}
    \log M(\varepsilon/c) \geq \log M^{\operatorname{loc}}(\varepsilon) \geq \log M(\varepsilon/c) - \log M(\varepsilon),
\end{align} 
where we denoted with $\log M(\varepsilon)$ the log cardinality of the maximal packing set of $K$ at a distance $\varepsilon$. The bounds \eqref{yang:barron:bound} follow from \cite{yang1999information}; actually \cite{yang1999information} only prove the bounds for the special case $c = 2$, but their results apply more generally.

Using the fact that the log cardinality of a maximal $\varepsilon$-packing set of the $\ell_1$ ball is given by $\log(\varepsilon^2 n)/\varepsilon^2$ for $\varepsilon \gtrsim 1/\sqrt{n}$, (otherwise it is $n$ if $\varepsilon \asymp 1/\sqrt{n}$ and $n\log\frac{1}{\varepsilon^2n}$ when $\varepsilon\lesssim 1/\sqrt{n}$ \cite{guedon2000euclidean, schutt1984entropy}), for $c$ large enough we have that 
\begin{align*}
    \log M(\varepsilon/c) - \log M(\varepsilon) \asymp \frac{\log (\varepsilon^2 n)}{\varepsilon^2} \asymp  \log M(\varepsilon/c).
\end{align*}

Hence, for $\varepsilon \gtrsim 1/\sqrt{n}$, the equation $\varepsilon^2/\sigma^2 \asymp \frac{\log (\varepsilon^2 n)}{\varepsilon^2}$ determines the minimax rate. Suppose that $\sigma$ is such that $\log((\sigma^2\log n)^{1/2}n) \asymp  \log n$, and $ (\sigma^2\log n)^{1/4} \gtrsim 1/\sqrt{n}$. Then setting $\varepsilon \asymp (\sigma^2\log n)^{1/4}$ solves the equation up to constant factors. This matches the example after Theorem 3 of \cite{donoho1994minimax} for $\sigma = 1/\sqrt{n}$. We conclude that
\begin{corollary}
The minimax rate for the $\ell_1$ ball is $(\sigma^2 \log n)^{1/2} \wedge 4$ for values of $\sigma$ such that $\log((\sigma^2\log n)^{1/2}n) \asymp  \log n$ and $ (\sigma^2\log n)^{1/4} \gtrsim 1/\sqrt{n}$.
\end{corollary}

It is worth pointing out that the orthogonal projection estimator, which works at a minimax rate in all of the aforementioned examples, fails to attain the rate for the $\ell_1$ ball \citep[see][e.g.]{zhang2013nearly}. On the other hand as we argue below the LSE works optimally for the $\ell_1$ ball. For an example of when both LSE and the projection estimator fail refer to Example 8 of \cite{zhang2013nearly}.

\subsection{Convex weak $\ell_p$ balls for $1 < p < 2$}

In this section we consider an example inspired by weak $\ell_p$ balls. Consider the quasi-norm $\|x\|_{p, \infty} = \max_{i \in [n]} i^{1/p} x_i^*$ on $\RR^n$ where $x_i^*$ denotes a decreasing rearrangement of $|x_1|, \ldots, |x_n|$, where $1 < p < 2$. Unfortunately $\|x\|_{p, \infty}$ is not a norm (so that its unit ball is not convex), but it admits an equivalent norm as follows. Consider 
\begin{align*}
    \|x\|_{p, \infty, *} = \max_{i \in [n]} i^{1/p} x_i^{**},
\end{align*}
where $x_i^{**} = i^{-1}\sum_{j = 1}^i x_j^*$. In this section we derive the minimax rate of the Gaussian sequence model for the convex set $K = \{x \in \RR^n : \|x\|_{p, \infty, *} \leq 1\}$. We will refer to $K$ as the convex weak $\ell_p$ ball. Using Theorem 2 of \cite{edmunds1998entropy} it is not too hard to see that the log cardinality of a maximal $\varepsilon$-packing set of $K$ (in Euclidean norm) is given by $\asymp \varepsilon^{-\frac{2p}{2-p}}\log(n \varepsilon^{\frac{2p}{2 - p}})$ for values of $\varepsilon \gtrsim n^{1/2 - 1/p}$. Observe that these bounds actually match the known bounds for $\ell_p$ balls \citep[see][e.g.]{schutt1984entropy}. Hence we can apply the same logic as in our $\ell_1$ example above, in that we can claim that for large enough $c$
\begin{align*}
    \log M(\varepsilon/c) - \log M(\varepsilon) \asymp \varepsilon^{-\frac{2p}{2-p}}\log(n \varepsilon^{\frac{2p}{2 - p}}) \asymp  \log M(\varepsilon/c),
\end{align*}
for $\varepsilon \gtrsim n^{1/2 - 1/p}$. Solving the equation $\frac{\varepsilon^2}{\sigma^2} \asymp \varepsilon^{-\frac{2p}{2-p}}\log(n \varepsilon^{\frac{2p}{2 - p}})$ gives, $\varepsilon \asymp \sigma^{\frac{4 - 2p}{4}}(\log n)^{\frac{2-p}{4}}$ given that $\sigma$ satisfies $\log(n \sigma^p (\log n)^{p/2}) \asymp \log n$. We conclude that 
\begin{corollary}
The minimax rate for the set $K$ as above is $\sigma^{2-p}(\log n)^{\frac{2-p}{2}} \wedge \diam(K)^2$ for values of $\sigma$ such that $\log(n \sigma^p (\log n)^{p/2}) \asymp \log n$ and $ \sigma^{\frac{4 - 2p}{4}}(\log n)^{\frac{2-p}{4}} \gtrsim n^{1/2 - 1/p}$.
\end{corollary}

\begin{remark}
Finally, let us remark that the same rate is valid for $\ell_p$ balls for $1 < p < 2$. This was first established in \cite{donoho1994minimax} (see their Theorem 3 for $\sigma = n^{-1/2}$). However, we would like to point out that the convex weak $\ell_p$ ball above is a larger set than the $\ell_p$ ball. This can be seen by the elementary inequality $\frac{\sum_{k = 1}^l |a_k|}{l} \leq \bigg(\frac{\sum_{k = 1}^l |a_k|^p}{l}\bigg)^{1/p}$ for any real numbers $\{a_k\}_{k = 1}^l$ and $p > 1$.
\end{remark}

\subsection{Bounds for a Bounded Convex Polytope with $N$ Vertices}

In this subsection we derive an upper bound on the minimax rate in the case when the set $K\subset \RR^p$ is a bounded convex polytope with $N$ vertices. Without loss of generality suppose $K$ is a polytope of diameter smaller than $1$, and it has exactly $N$ vertices. 

\subsubsection{Upper Bound}

By Maurey's empirical method, one can establish that $\log M(\varepsilon) \leq (C + 4C\varepsilon^2 N)^{\lceil 4/\varepsilon^2 \rceil}$ for some absolute constant $C$ (see Corollary 0.0.4 and Exercise 0.0.6 of \cite{vershynin2018high} and use the fact that the cardinality of a packing set of radius $2\varepsilon$ is smaller than the cardinality of a covering set of radius $\varepsilon$, see \eqref{packing:covering:bound} below). By \eqref{yang:barron:bound} we have $\log M^{\operatorname{loc}}(\varepsilon) \leq \log M(\varepsilon/c) \leq \lceil 4c^2/\varepsilon^2 \rceil \log (C + 4C\varepsilon^2 N/c^2)$. Thus an upper bound on the minimax rate is given by $\tilde \varepsilon^2 \wedge \diam(K)$ where $\tilde \varepsilon := \sup \bigg\{ \varepsilon : \frac{\varepsilon^2}{\sigma^2} \leq \lceil 4c^2/\varepsilon^2 \rceil \log (C + 4C\varepsilon^2 N/c^2)\bigg \}$. As illustrated in Section \ref{yang:barron:section:exmaple} this rate is in fact achieved for the ($\frac{1}{2}$-scaled) $\ell_1$ ball at least for a regime of $\sigma$ values. It is worth pointing out that since the upper bound based on Maurey's argument is nearly the same as that given by Sudakov minoration (see Corollary 7.4.4 in \cite{vershynin2018high}), it follows that the LSE will achieve (nearly) the same upper bound on the rate.

\subsubsection{Lower Bound}

In addition, we can show a matching lower bound for some convex polytopes as follows. Suppose there are $R \gtrsim N$ points $v_i \in K$ for $i \in [R]$ satisfying the following condition 
\begin{align}\label{raskutti:requirement}
    \|\sum_{i \in [R]} v_i \theta_i\| \geq \kappa_c \|\theta\| - f(n,p,R),
\end{align}
for any $\theta$ in the $2 \times \ell_1$ ball of $\RR^R$ for some small non-negative function $f(n,p,R)$ and for some positive constant $\kappa_c > 0$. By a sparse Varshamov-Gilbert lemma (see Lemma 10.12 \cite{foucart2013invitation}) one can find $L \geq \exp(c_1 k \log R/4 k)$ vectors $\{w_i\}_{i \in [L]}$ in the set $\{w \in \{0,1\}^R : \rho_H(w) = k\}$ where $\rho_H$ is the Hamming distance, such that $\rho_H(w_i, w_l) \geq c_2 k$. Now set $x_i = \sum_{j \in [R]} v_j w_{ij}/k$, and observe that $\|x_i - x_l\| = \|\sum_{j \in [R]} v_j (w_{ij} - w_{lj})/k\| \geq \|w_i -w_l\|/k - f(n,p,R) \geq \frac{\sqrt{c_2}}{\sqrt{k}} - f(n,p,R)$. It follows that for $k = \frac{c_2}{(\varepsilon + f(n,p,R))^2}$ the set $\{x_i\}_{i \in [L]}$ is an $\varepsilon$ packing set. Thus $\log M(\varepsilon) \geq c_1 \frac{c_2}{(\varepsilon + f(n,p,R))^2} \log \frac{R(\varepsilon + f(n,p,R))^2}{4 c_2}$. To simplify the calculations suppose $\varepsilon \gtrsim f(n,p,R)$, to obtain that $\log M(\varepsilon) \geq  \frac{c_1'}{\varepsilon^2} \log \frac{R \varepsilon^2}{c'_2}$. Now one can use \eqref{yang:barron:bound} coupled with the upper bound on $\log M(\varepsilon)$ via Maurey's argument above, to claim that for a sufficiently large $c$ the $\log M^{\operatorname{loc}}(\varepsilon) \gtrsim \frac{c_1'}{\varepsilon^2} \log \frac{R \varepsilon^2}{c'_2}$ for $\varepsilon \gtrsim f(n,p,R)$. It follows that if the solution $\varepsilon^*$ to the equation $\frac{\varepsilon^2}{\sigma^2} \asymp \frac{\log \frac{R \varepsilon^2}{c'_2}}{\varepsilon^2}$ is $\gtrsim f(n,p,N)$ $\varepsilon^{*2}\wedge \diam(K)$ is a lower bound on the rate. Further since $R \gtrsim N$ then the lower and upper bounds would match provided that $\varepsilon^* \gtrsim f(n,p,R)$.

One instance when such a scenario can appear in practice is when $K = X \beta$ for $\beta \in \ell^p_1(1)$, where we denoted the unit $\ell_1$ ball in $\RR^p$ with $\ell^p_1(1)$. Assuming that $\max_{i \in [p]}\|X_i\| \leq 1$, it follows that $K$ is a symmetric polytope with at most $N \leq 2p$ vertices. In this case one can see that the calculations above recover the bounds given in Theorems 3 and 4 in \cite{raskutti2011minimax} for the $\ell_1$ ball in the case when $\sigma \asymp \frac{1}{\sqrt{n}}$. Here the quantity $f(n,p,N)$ can be taken as $f(n,p,N) \lesssim \sqrt{\frac{\log p}{n}}$. One example of a matrix $X$ that satisfies condition \eqref{raskutti:requirement} with high probability is if the rows of $X$ consist of i.i.d. $ N(0,\mathbb{I}_p)/\sqrt{C'n}$ for a sufficiently large $C'$ variables. Then with high probability it can be shown the columns of $X$ are bounded in $\ell_2$ norm (see Appendix I \cite{raskutti2011minimax}), and also by Proposition 1 of \cite{raskutti2011minimax}  \eqref{raskutti:requirement} is satisfied by $R = p \gtrsim N$ points.

\subsection{Cartesian Product of Sets}

In this section we consider the example when $K = K_1 \times K_2$ is a Cartesian product of two closed bounded convex sets. Intuitively it should be clear that if one has a minimiax rate optimal estimator on $K_1$ and a minimax rate optimal estimator on $K_2$ by running them separately one will obtain at most twice the maximum of the two rates. On the other hand, for the lower bound it is clear that either of the two minimax rates are lower bounds on the minimax rate over $K$. Below we make this intuition precise by using local packing entropy calculations.

\subsubsection{Upper Bound}

We begin by reminding the reader that 
\begin{align}\label{packing:covering:bound}
    M(2\delta, S) \leq N(\delta, S) \leq M(\delta, S),
\end{align} 
where $M$ and $N$ denote the maximal packing and minimum covering numbers of the (totally) bounded set $S \subset \RR^n$ in Euclidean norm, and the $\delta$ (or $2\delta$) indicates at what distance we are packing or covering (see \cite[Lemma 5.5][e.g.]{wainwright2019high}).

Consider now a fixed point $(x^\circ, y^\circ) \in K$ such that $x^\circ \in K_1$ and $y^\circ \in K_2$ are arbitrary points. Let $N_1$ be a minimal covering of the set $B(x^\circ, \varepsilon) \cap K_1$ and $N_2$ be a minimal covering of the set $B(y^\circ, \varepsilon) \cap K_2$ at a distance $\varepsilon/4c$. Put $\tilde N = N_1 \times N_2$. Consider $N' = \Pi_{B((x^\circ, y^\circ), \varepsilon) \cap K} \tilde N$ which is the projection of $\tilde N$ onto the closed convex set $B((x^\circ, y^\circ), \varepsilon) \cap K$. We will show that $N'$ is a covering of $B((x^\circ, y^\circ), \varepsilon) \cap K$. First let us verify that for a point $(x,y) \in N'$ we have $\|(x,y) - (x^\circ, y^\circ)\|\leq \varepsilon$. This is so simply by the fact that we projected $\tilde N$ on the set $B((x^\circ, y^\circ), \varepsilon) \cap K$. Now for an arbitrary point $(\bar x,\bar y) \in B((x^\circ, y^\circ), \varepsilon) \cap K$ let us find $(x', y') \in N'$ such that $\|(\bar x,\bar y) - (x', y')\|$ is small. Let $\tilde x$ be the point closest to $\bar x$ from $N_1$ and similarly let $\tilde y$ be the point closest to $\bar y$ from $N_2$. Define $ (x', y') =  \Pi_{B((x^\circ, y^\circ), \varepsilon) \cap K}(\tilde x, \tilde y) \in  N'$. We have
\begin{align*}
    \|(\bar x,\bar y) - (x', y')\| \leq \|(\bar x,\bar y) - (\tilde x, \tilde y)\| \leq \|\bar x - \tilde x\| + \|\bar y - \tilde y\| \leq \frac{\varepsilon}{2c},
\end{align*}
where in the above the first inequality follows by the fact that $(\bar x,\bar y) \in B((x^\circ, y^\circ), \varepsilon) \cap K$ and the projection does not increase the distance between the point and any point in the set $B((x^\circ, y^\circ), \varepsilon) \cap K$, and the last inequality is true because $\|\bar x - x^\circ\| \leq \varepsilon$ and similarly $\|\bar y -  y^\circ\| \leq \varepsilon$ and the definitions of $N_1$ and $N_2$. Now using \eqref{packing:covering:bound}, we conclude that $\log M^{\operatorname{loc}}_{K}(\varepsilon) \leq 2 (\log M^{\operatorname{loc}}_{K_1}(\varepsilon, \varepsilon/4c)\vee  \log M^{\operatorname{loc}}_{K_2}(\varepsilon, \varepsilon/4c))$,
where we denoted with $M^{\operatorname{loc}}_{K_1}(\varepsilon, \varepsilon/4c)$ the local packing entropy of $K_1$ of radius $\varepsilon$ at a distance $\varepsilon/4c$ (instead of $\varepsilon/c$) and similarly for the term $M^{\operatorname{loc}}_{K_2}(\varepsilon, \varepsilon/4c)$.

\subsubsection{Lower Bound}

In this section we establish an lower bound on the rate. Let $(x^\circ, y^\circ) \in K$ be a point where $x^\circ \in K_1$ and $y^\circ \in K_2$ are arbitrary points. Consider two maximal packing sets $M_1$ and $M_2$ of $B(x^\circ, \varepsilon/2) \cap K_1$ and $B(y^\circ, \varepsilon/2) \cap K_2$ at a distance $\sqrt{2}\varepsilon/c$. Let $M$ be a maximal packing set of $B((x^\circ, y^\circ), \varepsilon) \cap K$ at a distance $\varepsilon/c$. We claim that 
\begin{align}\label{log:cardiniality:ineq}
    \log |M| \geq \log |M_1| + \log |M_2|.
\end{align} This is so since the set $M' = M_1 \times M_2$ forms a packing set of $B((x^\circ, y^\circ), \varepsilon) \cap K$. To see this we first verify that for all $(x,y) \in M'$ we have $\|(x,y) - (x^\circ,y^\circ)\| \leq \varepsilon$. This is true since $\|(x,y) - (x^\circ,y^\circ)\| \leq \|x- x^\circ\| + \|y- y^\circ\| $, and the requirements for the points in $M_1$ and $M_2$. Next for any two distinct points in $(x,y), (x',y') \in M'$ (i.e., $x\neq x'$ and/or $y \neq y'$) we have $\|(x,y) - (x',y')\|\geq \frac{\|x - x'\| + \|y - y'\|}{\sqrt{2}} \geq \varepsilon/c$. This finishes the proof. Next, \eqref{log:cardiniality:ineq} implies that 
\begin{align*}
    \log M^{\operatorname{loc}}_{K}(\varepsilon) & \geq \log M^{\operatorname{loc}}_{K_1}(\varepsilon/2, \sqrt{2}\varepsilon/c) \vee \log M^{\operatorname{loc}}_{K_2}(\varepsilon/2, \sqrt{2}\varepsilon/c)\\
    & \geq \log M^{\operatorname{loc}}_{K_1}(\varepsilon, 2\sqrt{2}\varepsilon/c) \vee \log M^{\operatorname{loc}}_{K_2}(\varepsilon, 2\sqrt{2}\varepsilon/c),
\end{align*} where as in the upper bound we denoted with $M^{\operatorname{loc}}_{K_1}(\varepsilon/2, \sqrt{2}\varepsilon/c)$ the local packing entropy of $K_1$ of radius $\varepsilon/2$ (instead of $\varepsilon$) at a distance $\sqrt{2}\varepsilon/c$ and similarly for the term $M^{\operatorname{loc}}_{K_2}(\varepsilon/2, \sqrt{2}\varepsilon/c)$, and in the last inequality we used Lemma \ref{simple:lemma:monotone}.

Combining the results from the previous two subsections, and the fact our results are robust to changes in $c$, i.e., to selecting $c$ to be slightly bigger or smaller sufficiently large constant we conclude that:
\begin{corollary}
The minimax rate up to constant factors is given by $\varepsilon^{*2} \wedge \diam(K)^2$ where \begin{align}\label{cartesian:product:of:sets:rate}
    \varepsilon^* = \sup\{\varepsilon : \varepsilon^2/\sigma^2\leq \log M^{\operatorname{loc}}_{K_1}(\varepsilon)\vee \log M^{\operatorname{loc}}_{K_2}(\varepsilon)\}.
\end{align}
\end{corollary}

\begin{remark}
Let us remark that the corollary above can give rise to many examples where the minimax rate can be quantified with more interpretable quantities than the local entropies, for instance when $K_1$ and $K_2$ are an ellipse and a hyperrectangle. Of course this bound also extends to the case when $K = \prod_{j = 1}^k K_j$ as long as the number of sets $k$ remains fixed, i.e., it does not scale with $n$ (or $\sigma$). Finally we remark that the same logic shows that if one has a set $K$ which is a direct sum $K = K_1 \oplus K_2$, where $K_1 \perp K_2$ are orthogonal bounded and closed convex sets the minimax rate on the sum would be given by $\varepsilon^{*2} \wedge \diam(K)^2$ where $\varepsilon^*$ is determined via equation \eqref{cartesian:product:of:sets:rate}. This is so since for any two points $z = x + y, z' = x' + y' \in K$ where $x,x' \in K_1$ and $y,y' \in K_2$ we have
\begin{align*}
(\|x - x'\| + \|y - y'\|)^2 & \geq \|x + y - (x' + y')\|^2 \\
& = \|x - x'\|^2 + \|y - y'\|^2 \\
&\geq  (\|x - x'\| + \|y - y'\|)^2/2,
\end{align*}
so that the same proof as above will apply. 

\end{remark}

\section{Adaptivity and Admissibility up to a Universal Constant} \label{adaptivity:sec}

In this section we argue that the estimator constructed in Algorithm \ref{test} is adaptive to the true point. It will be beneficial to define local entropy in a slightly different manner than before.

\begin{definition}\label{def:local:entropy:at:a:point} Let $\theta \in K$ be a point. Consider the set $B(\theta, \varepsilon) \cap K$. For $\theta \in K$ let $M(\theta, \varepsilon, c) := M(\varepsilon/c, B(\theta, \varepsilon) \cap K)$ denote the largest cardinality of an $\varepsilon/c$ packing set in $B(\theta, \varepsilon) \cap K$. 
\end{definition}
\begin{remark} We would like to underscore the fact that Definition \ref{def:local:entropy:at:a:point} does not take a supremum over all points in the set $K$. This small but key difference is what enables us to formalize the adaptive result below.

\end{remark}
We first prove the following lemma.
\begin{lemma} Suppose $\nu$ and $\mu$ are two points in $K$ such that $\|\nu - \mu\| < \delta$. Then $M(\nu, \varepsilon, c) \leq M(\mu, 2\varepsilon, 2c)$ for any $\varepsilon > \delta$.
\end{lemma}
\begin{proof} It suffices to show that $B(\nu, \varepsilon) \cap K \subset B(\mu, 2\varepsilon) \cap K$. We will show directly that $B(\nu, \varepsilon) \subset B(\mu, 2\varepsilon)$. Take any point $x \in B(\nu, \varepsilon) $. By the triangle inequality $\|x-\mu\| \leq \|x - \nu\| + \delta \leq 2\varepsilon$ since we are assuming $\delta < \varepsilon$. This completes the proof. 
\end{proof}

Using the above lemma, one can modify the proof of Theorem \ref{upper:bound:rate} to arrive at the following adaptive version of the result.

\begin{theorem} The estimator from Algorithm \ref{test} returns a vector $\nu^*$ which satisfies the following property
\begin{align*}
    \mathbb{E} \|\mu - \nu^*\|^2 \leq \bar C \varepsilon^{*2},
\end{align*}
for some universal constant $\bar C$, where $\varepsilon^* = \varepsilon_{J^*}$ and $J^*$ is the maximal $J \geq 1$ such that $\varepsilon_J := \frac{d (c/2 - 3)}{2^{J -2}c}$ satisfies
\begin{align*}
    \frac{\varepsilon_J^2}{\sigma^2} > 16\log M\bigg(\mu, 2\varepsilon_J \frac{c}{(c/2-3)}, 2c\bigg) \vee 16\log 2,
\end{align*}
of $J^* = 1$ if no such $J$ exists.
\end{theorem}

The main thing that needs to be modified is the local entropy in the bound \eqref{important:telescoping:bound:which:is:true}. We omit the details.

The final remark of this section is to observe that due to the minimaxity of the estimator in Algorithm \ref{test}, we have that it is admissible up to a universal constant. This is a trivial observation. For any estimator $\hat \nu(Y)$, there exists a point $\theta \in K$ such that 
\begin{align*}
    \mathbb{E} \|\hat \nu(Y) - \theta\|^2 \geq \bar c \varepsilon^{*2} \wedge d^2,
\end{align*}
where $\bar c$ is a universal constant. On the other hand we know that $\mathbb{E} \|\nu^*(Y) - \theta\|^2 \leq 
\bar C \varepsilon^{*2} \wedge d^2$ where $\bar C$ is another universal constant. Hence the conclusion. 




\section{Unbounded Sets with Known $\sigma^2$} \label{unbounded:sets:section}


In this section we generalize the results of Section \ref{main:results:section} to the unbounded case with known $\sigma^2$. A new algorithm is needed which runs multiple bounded algorithms and ``aggregates'' them in a way similar to how we constructed the bounded case algorithm. The only place where knowledge of $\sigma^2$ is used is to ``split'' the sample into two independent samples.

\subsection{Lower Bound}

Note that for unbounded convex sets, the lower bound remains valid. Namely, as long as, $\log M^{\operatorname{loc}}(\varepsilon)\allowbreak > 4 \varepsilon^2/\sigma^2 \vee 4 \log 2$ the minimax risk is at least $\varepsilon^2/8 c^2$. Observe also, that for a sufficiently large $c$ the term $4 \log 2$ does not have effect on the lower bound. This is so since any unbounded convex set in $\mathbb{R}^n$ contains a ray \cite[see Lemma 1 Section 2.5][e.g.]{grunbaum2013convex}, and therefore, one can position a ball of radius $\varepsilon$ on that ray so that part of the ray with length $2\varepsilon$ is fully in the ball. Then one can put $\exp(4\log 2)$ balls of radius $\varepsilon/c$ on that ray centered at equispaced points, which will ensure that $\log M^{\operatorname{loc}}(\varepsilon) > 4 \log 2$ for any $\varepsilon$.


\subsection{Upper Bound} \label{lower:boubd:section:no:constnat}

In this section we describe an algorithm for unbounded convex sets, and show it achieves the minimax rate. We start with a simple lemma. For simplicity we will assume that the given set $K$ is closed, but we remark how to fix our argument for sets that are not necessarily closed in Remark \ref{not:closed:extension}.
\begin{lemma}
For two convex sets $S,S'$ satisfying $S' \subset S$, we have that $M^{\operatorname{loc}}_{S'}(\varepsilon) \leq M^{\operatorname{loc}}_S(\varepsilon)$ for any $\varepsilon > 0$.
\end{lemma}
\begin{proof}
Since for any $\theta \in S'$ we have $B(\theta, \varepsilon) \cap S' \subset B(\theta, \varepsilon) \cap S$ the proof is complete.
\end{proof}

We first use the knowledge of $\sigma^2$ to ``split'' the sample. To this end let us draw $\eta \sim N(0,\mathbb{I}\sigma^2)$ independently from the observed data $Y$. Consider the variables $\tilde Y^1 = Y + \eta$ and $\tilde Y^2 = Y - \eta$. These variables are independent. Take any fixed point $\nu \in K$. We consider balls centered at $\nu$ with different radiuses $B(\nu,1)\cap K$,$B(\nu,2)\cap K$, $\ldots$, $B(\nu,2^m)\cap K$, $\ldots$ and every time compute the estimator from Algorithm \ref{test} using $\tilde Y^1$ as the ``$Y$ value''. Denote these estimators with $\{\nu_m\}_{m = 1}^{\infty}$. Note that since $K$ is closed all of these estimators are proper (i.e. they output values in $K$). The intuition for constructing these, is that for large enough $m$ these estimators will have good properties as $\mu$ will belong to the set $B(\nu,2^m)\cap K$. 
We have the following lemma regarding the sequence of estimators $\nu_m$. 

\begin{lemma}\label{compact:set:estimators} All estimators $\nu_m$ lie in a compact set.
\end{lemma}

\begin{remark}
We would like to remark that this compact set depends on $\tilde Y^1$ and the true point $\mu$. This is not an issue for our analysis since the two samples $\tilde Y^1$ and $\tilde Y^2$ are independent by construction, hence we may consider the first sample as ``frozen''.
\end{remark}
\begin{proof} For brevity throughout the proof we denote $\tilde Y^1$ with $Y$. Let $P_{K} Y $ denote the projection of $Y$ onto the set $K$ (this is a well defined operator since $K$ is assumed to be closed). At some point the radius $2^N$ will be so big that $P_{K} Y$ will be in the set $B(\nu,2^N) \cap K$. From there on, i.e. $m \geq N$, we will argue that the estimators $\nu_m$ will be close to the point $P_{K} Y$. The first packing set is at distance $\frac{d}{2(C+ 1)}$ where $d \leq 2^{m + 1}$ and $C$ is the constant from Algorithm \ref{test} (such that $2 (C + 1) = c$). Let $x = \|Y - P_{K} Y\|$. For any point $\nu \in K$ we have $\sqrt{x^2 + \|\nu - P_{K} Y\|^2} \leq \|\nu - Y\| \leq x + \|\nu - P_{K} Y\|$, where the first inequality follows by the cosine theorem, and the second one from the triangle inequality. On the other hand the closest point $\bar \nu$ from the packing set to $P_{K} Y$ satisfies $\|\bar \nu - P_{K} Y\| \leq \frac{d}{2 (C+1)}$, and therefore 
\begin{align*}
    \|\bar \nu - Y\| \leq x + \|\bar \nu - P_{K} Y\| \leq x + \frac{d}{2 (C+1)}.
\end{align*}

Take $\hat \nu$ to be the closest point to $Y$. We then have 
\begin{align*}
    \sqrt{x^2 + \|\hat \nu - P_{K} Y\|^2} \leq \|\hat \nu - Y\|  \leq \|\bar \nu - Y\| \leq x +  \frac{d}{2 (C+1)}.
\end{align*}
It follows that 
\begin{align*}
    \|\hat \nu - P_{K} Y\|^2 \leq 2 x \frac{d}{2 (C+1)} +  \bigg(\frac{d}{2 (C+1)}\bigg)^2 \leq 3 \bigg (\frac{d}{2 (C+1)}\bigg)^2,
\end{align*}
assuming that $x \leq \frac{d}{2 (C+1)}$. Since $C \geq 2$ this implies that $ \|\hat \nu - P_{K} Y\| \leq \frac{d}{2}$, and thus the point $P_{K} Y$ will be in the chosen ball for the second step. We can continue this logic until, $x \geq \frac{d}{2^k (C+1)}$. At this point we know that the estimator will be within distance $\frac{d}{2^{k-2}}$ of the central point, which is at distance at most $\frac{d}{2^{k-1}}$ from $P_{K} Y$, so that the final estimator will be at distance at most $\frac{3 d}{2^{k-1}} \leq 6(C+1)x$ from $P_{K} Y$. This completes the proof that all estimators will be on a compact set since the initial ones fall into a ball of radius $2^N$ and are also in a compact set. 
\end{proof}
\begin{remark}
The lemma above extends to the case where $K$ is not closed. The only thing that needs to be modified in the proof is that $P_K Y$ should be interpreted as $P_{\overline{K}} Y$ where as usual $\overline{K}$ is the closure of $K$.
\end{remark}

Define $\tilde C = \frac{c}{4} - 1$, where $c$ is the local packing constant from Definition \ref{local:entropy:def}. Once we have established Lemma \ref{compact:set:estimators}, we can proceed to propose Algorithm \ref{test:v2}. As we mentioned previously, this algorithm runs multiple bounded algorithms and ``aggregates'' them in a way similar to how Algorithm \ref{test} works. 

\begin{algorithm}
\SetKwComment{Comment}{/* }{ */}
\caption{Upper Bound Algorithm (Unbounded Case)}\label{test:v2}
\KwInput{A sequence of estimators $\mathcal{E} := \{\nu_m\}_{m \in \mathbb{N}} \subset K$; $d$ the diameter of $\mathcal{E}$ which is bounded by Lemma \ref{compact:set:estimators}; $\nu^* \in \mathcal{E}$ an arbitrary point.}
$k \gets 1$\;
$\Upsilon \gets [\nu^*]$\;
\While{TRUE}{
    Take a $\frac{d}{2^{k + 1}(\tilde C+1)}$ maximal\footnotemark[4] packing set $M_k$ of the set $B\big(\nu^*, \frac{d}{2^{k-1}}\big) \cap \mathcal{E}$ \Comment*[r]{The packing sets should be constructed in a special way as described in the proof of Theorem \ref{measurability:theorem:unbounded:case} to ensure measurability}
    $\nu^* \gets \argmin_{\nu \in M_k} \|\tilde Y^2 - \nu\| $ \Comment*[r]{Break ties by taking the point with smallest index in $\mathcal{E}$}
    $\Upsilon$.append$(\nu^*)$\;
    $k \gets k + 1$\;
}
\Return{$\nu^*$}  \Comment*[r]{Observe that by definition $\Upsilon$ forms a Cauchy sequence, so $\nu^*$ can be understood as the limiting point of that sequence.}
\end{algorithm}
\footnotetext[4]{It is not important for the packing set to be maximal as long as it is a covering set. See Theorem \ref{measurability:theorem:unbounded:case} for a specification of how to construct these sets to ensure measurability.}

Before we proceed with the proof of why Algorithm \ref{test:v2} works, we will show that the estimator produced by it is measurable. We have

\begin{theorem}\label{measurability:theorem:unbounded:case} We have that $\nu^* : \RR^n \times \RR^n \mapsto \RR^n$ is a measurable function (with respect to the Borel $\sigma$-field). As a consequence $\nu^*(Y, \eta)$ is a random variable.
\end{theorem}

\begin{proof} 
We will show that each element in the sequence $\Upsilon_j$ is measurable. Since they form a Cauchy sequence their limit will also be measurable by an argument similar to the one in Theorem \ref{measurability:theorem}. Throughout the proof, so as to not overburden notation, for the most part we will suppress the dependence of the estimators $\nu_m$ on $\tilde y^1 = y + \eta$ and will simply write $\nu_m$. We will also suppress the dependence of $\Upsilon_j$ on $y$ and $\eta$.

We will select a packing set greedily starting with the minimum index that belongs to the ball on the $k$-th step, then carving a ball out centered at that minimum index, and next considering the minimum index that is in the bigger ball but is out of the carved out ball and so on. We will first show that $\Upsilon_1$ is measurable. For $\Upsilon_1$ the big ball on the $1$-st step contains all estimators $\nu_m$ hence we start from $\nu_1$. We will show that the event $\Upsilon_1 = \nu_j$ is a measurable event, and since as we know from before each $\nu_j$ is measurable, and the identity $(y, \eta: \Upsilon_1 \in B) = \cup_j (y, \eta : \Upsilon_1 = \nu_j) \cap (y, \eta : \nu_j(y + \eta) \in B)$ for any hyperrectangle $B$ we will have that $\Upsilon_1$ is measurable. We will now give a little details about the measurability of the event $(y, \eta : \nu_j(y + \eta) \in B)$. For $(y, \eta : \nu_j(y + \eta) \in B) = (y,\eta : y + \eta \in B')$ for some Borel set $B'$ by the measurability of $\nu_j$. This is a Borel set since the function $(y,\eta) \mapsto y + \eta$ is continuous and hence measurable.

Let us call the index set of the chosen packing (according to the strategy described above), ``the index set''. We then have the identity:
\begin{align*}
    \{y, \eta: \Upsilon_1 = \nu_j\} &= \cup_{S: j \in S, |S| \leq M^{\operatorname{loc}}(r)} \bigg(\{y,\eta : S \mbox{ is the index set}\} \cap \\
    &\cap_{i \in S}\{y, \eta: \|\nu_j - \tilde y^2\| \leq \|\nu_i - \tilde y^2\|\} \cap_{i \in S, i \leq j} \{y, \eta: \|\nu_i - \tilde y^2\| \neq \|\nu_j - \tilde y^2\|\}\bigg),
\end{align*}
where we put for brevity $r = d/(4(\tilde C + 1))$ and $\tilde y^2 = y - \eta$. Let $S = (s_1,s_2,\ldots, s_m)$ (note that $s_1 = 1$ always has to belong in $S$). The above events in the latter two intersections are measururable since for two measurable functions $X$ and $Y$ the events $X \leq Y$ and $X \neq Y$ are measurable, the function $\|\cdot\|$ is continuous hence measurable, the sum (difference) of two measurable functions is measurable, and the maps $\nu_j(y + \eta)$ and $y - \eta$ are measurable (as we argued earlier and by continuity). Now, the event that $S$ is the index set is
\begin{align*}
    \{y, \eta: S \mbox{ is the index set}\} & = \cap_{k = 2}^{s_2 - 1}\{y, \eta: \|\nu_1 - \nu_k\| \leq r \} \cap \{y, \eta: \|\nu_1 - \nu_{s_2}\| > r\} \cap \\
    & \cap_{k = s_2 + 1}^{s_3-1} \{y, \eta: \|\nu_1 - \nu_k\| \leq r\} \cup \{\omega: \|\nu_{s_2} - \nu_k\| \leq r\} \\
    & \cap \{y, \eta:\|\nu_1 - \nu_{s_3}\| > r\} \cap \{y, \eta: \|\nu_{s_2} - \nu_{s_3}\| > r\} \cap\\
    &\ldots\\
    &\cap_{k \geq s_m + 1} (\{y, \eta: \|\nu_1 - \nu_k\| \leq r\} \cup \{y, \eta: \|\nu_{s_2} - \nu_k\| \leq r\} \cup \\
    & \ldots \cup\{y, \eta: \|\nu_{s_m} - \nu_k\| \leq r\}),
\end{align*}
which is clearly measurable (by continuity of $\|\cdot\|$, and the fact that the difference of measurable functions is measurable). This completes the proof that $\Upsilon_1$ is measurable. We will now argue that $\Upsilon_2$ is also measurable using the same trick. Observe that the identity:
\begin{align*}
    \{y, \eta: \Upsilon_2 = \nu_j\} &= \cup_{S: j \in S, |S| \leq M^{\operatorname{loc}}(r)} \bigg(\{y, \eta: S \mbox{ is the index set}\} \cap \\
    &\cap_{i \in S}\{y, \eta: \|\nu_j - \tilde y^2\| \leq \|\nu_i - \tilde y^2\|\} \cap_{i \in S, i \leq j} \{y, \eta: \|\nu_i - \tilde y^2\| \neq \|\nu_j - \tilde y^2\|\}\bigg),
\end{align*}
continues to hold for $\Upsilon_2$ with the only difference that $r = d/(8(\tilde C + 1))$. We will now show that the event $\{y, \eta: S \mbox{ is the index set}\}$ continues to be measurable for $\Upsilon_2$. We have
\begin{align*}
     \{y, \eta: S \mbox{ is the index set}\} & = \cap_{k = 1}^{s_1 - 1}\{ y, \eta: \|\Upsilon_1 - \nu_k\| > d/2\} \cap \{y, \eta: \|\Upsilon_1 - \nu_{s_1}\| \leq d/2\}\\
     & \cap_{k = s_1 + 1}^{s_2 - 1} (\{ y, \eta: \|\Upsilon_1 - \nu_k\| > d/2\} \cup \{\omega: \|\nu_{s_1} - \nu_k\| \leq r \}) \\
     &\cap (\{ y, \eta: \|\Upsilon_1 - \nu_{s_2}\| \leq d/2\} \cap \{y, \eta: \|\nu_{s_1} - \nu_{s_2}\| > r\}) \cap \\
    &\ldots\\
    &\cap_{k \geq s_m + 1} (\{ y, \eta: \|\Upsilon_1 - \nu_k\| > d/2\} \cup \{y, \eta: \|\nu_1 - \nu_k\| \leq r\} \\
    & \cup \{y, \eta: \|\nu_{s_2} - \nu_k\| \leq r\} \cup \ldots \cup\{y, \eta: \|\nu_{s_m} - \nu_k\| \leq r\}),
\end{align*}

Clearly, all of the above are measurable events, and therefore $\Upsilon_2$ is measurable. Proving that all subsequent $\Upsilon_j$ are measurable is the same as proving that $\Upsilon_2$ is measurable which completes the proof.

\end{proof}

Next we prove a modification of Lemma \ref{most:importnant:lemma}. The setting is as follows. We are given $M$ points $\nu_1, \ldots, \nu_M \in K$ such that $\min \|\nu_i - \mu\| \leq \rho$.

\begin{lemma}\label{most:importnant:lemma:v2}Let $i^* = \argmin_i \|\tilde Y^2 - \nu_i\|$. We will show that the closest point to $\tilde Y^2$, $\nu_{i^*}$ satisfies
\begin{align*}
    \mathbb{P}(\|\nu_{i^*} - \mu\| > (C + 1)\rho) \leq M \exp(-(C - 2)^2 \rho^2/(16\sigma^2)),
\end{align*}
for any fixed $C > 2$.
\end{lemma}

\begin{proof}
Define the intermediate random variable
\begin{align*}
    T_i  = \begin{cases}
         \max_{j \in [M]}  \|\nu_i - \nu_j\|, \mbox{ s.t. } \|\tilde Y^2 - \nu_i\| - \|\tilde Y^2 - \nu_j\| \geq 0,\|\nu_i - \nu_j\| > C\rho\\
         0, \mbox{if no such $j$ exists},
    \end{cases}
\end{align*}
Without loss of generality assume that $\|\mu - \nu_i\| \leq \rho$. Next, we have that
\begin{align*}
    \mathbb{P}(\|\nu_{i^*} - \mu\| > \rho + C\rho) & \leq \mathbb{P}(i^* \in \{j: \|\nu_j - \nu_i\| > C\rho\})\\
    & \leq P(T_i > 0),
\end{align*}
where the first inequality follows by the triangle inequality and the second because if $i^* \in \{j: \|\nu_j - \nu_i\| \geq C\rho\}$ we have $T_{i} \geq \|\nu_i - \nu_{i^*}\| > C\rho$. But
\begin{align*}
    \mathbb{P}(T_i > 0) & = \mathbb{P}(\exists j: \|\nu_j - \nu_i\| > C\rho \mbox{ and } \|\tilde Y^2 - \nu_i\| - \|\tilde Y^2 - \nu_j\| \geq 0)
    \\& \leq M \exp(-(C-2)^2 \rho^{2}/(16\sigma^2)),
\end{align*}
by Lemma \ref{important:lemma} (here we used the fact that $\xi_i - \eta_i \sim N(0,2\sigma^2)$). This is what we wanted to show.
\end{proof}

\begin{theorem}\label{unbounded:thm} The estimator from Algorithm \ref{test:v2} returns a vector $\nu^*$ which satisfies the following property
\begin{align*}
    \mathbb{E} \|\mu - \nu^*\|^2 \leq \bar C \varepsilon^{*2},
\end{align*}
for some universal constant $\bar C$, where $\varepsilon^*$ is the smallest solution to
\begin{align}\label{inequality:unbounded:case}
    \frac{\varepsilon^2}{\sigma^2} > 32\log M^{\operatorname{loc}}\bigg(\varepsilon \frac{c}{c/2-3}\bigg) \vee 32\log 2.
\end{align}
We remind the reader that $c$ is the constant from the definition of local entropy, which is assumed to be sufficiently  large.
\end{theorem}

\begin{remark}\label{remark:no:constnant:upper:bound}
For $c$ large enough inequality \eqref{inequality:unbounded:case} is equivalent to simply
\begin{align*}
    \frac{\varepsilon^2}{\sigma^2} > 32\log M^{\operatorname{loc}}\bigg(\varepsilon \frac{c}{c/2-3}\bigg),
\end{align*}
since one can always take the center of the ball lying on an infinite ray (which exists \cite[see Lemma 1 Section 2.5][e.g.]{grunbaum2013convex}), and then there will exist at least $\exp(\log 2)$ equispaced points on that ray.
\end{remark}

\begin{remark}
Note that the expected value in \eqref{inequality:unbounded:case} is taken with respect to both $\xi$ and $\eta$. It is clear by Jensen's inequality, that the estimator $\EE_{\eta} \nu^*(Y,\eta)$ satisfies
\begin{align*}
    \EE_{\xi} \|\mu - \EE_{\eta} \nu^*(Y,\eta)\|^2 \leq \mathbb{E} \|\mu - \nu^*\|^2 \leq \bar C \varepsilon^{*2}.
\end{align*}
Note that since $\EE_{\eta} \nu^*(Y,\eta) = \EE [\nu^*(Y,\eta) | Y]$ it is a measurable function of the data $Y$, and therefore achieves the minimax rate as shown in Proposition \ref{proposition:minimax:rate:unbounde:set}.
\end{remark}


\begin{proof}
Let $\rho = \inf_{j} \|\mu - \nu_j\|$, and let $\bar \nu$ be a limiting point of $\nu_j$ such that $\rho = \|\mu - \bar \nu\|$. Note that $\rho$ is fixed given $\tilde Y^1$. We know that for the $N$-th estimator where $N$ is such that $2^N \geq \|\mu - \nu\|$ we have that the conditions of Theorem \ref{upper:bound:rate} are fulfilled and by \eqref{important:prob:bound} therefore
\begin{align}\label{to:modify:thm:unbounded}
    \mathbb{P}(\rho > 2 \kappa x) \leq \mathbb{P}(\|\mu - \nu_N\|  > 2 \kappa x) \leq \underline C \exp(-C'x^2/\sigma^2) \mathbbm{1}(J^* > 1),
\end{align}
which holds true for $x \geq \varepsilon^*$, where $\varepsilon^* = \varepsilon_{J^*} = \frac{(C-2)\operatorname{diam}(B(\nu, 2^N) \cap K)}{(C+1) 2^{J^* - 1}}$, and where $J^*$ is the maximum $J$ selected so that $\frac{\varepsilon_J^{2}}{2\sigma^2} > 16\log M_{B(\nu, 2^{N})\cap K}^{\operatorname{loc}}\bigg(\varepsilon_J \frac{2 (C+1)}{(C-2)}\bigg) \vee 16\log 2$ of $J^* = 1$ if such $J$ does not exist. Here we have $2\sigma^2$ in the denominator since $\xi_i + \eta_i \sim N(0, 2\sigma^2)$.

For any $J$ such that $\frac{d}{2^{J + 1} (\tilde C + 1)} \geq \rho$ by Lemma \ref{most:importnant:lemma:v2} we have the following bound (recall that $c = 4(\tilde C+1)$ where $c$ is the constant from the definition of local packing entropy):
\begin{align*}
    \MoveEqLeft \mathbb{P}\bigg(\|\bar \nu - \Upsilon_J\| > \frac{d}{2^{J-1}} \bigg | \|\bar \nu - \Upsilon_{J-1}\|\leq \frac{d}{2^{J-2}}, \tilde Y^1, \Upsilon_{J-1}\bigg)
    \\&\leq\mathbb{P}\bigg(\|\bar \nu - \Upsilon_J\| > \rho + (\tilde C+1)(\frac{d}{2^{J}(\tilde C + 1)} + \rho) \bigg | \|\bar \nu - \Upsilon_{J-1}\|\leq \frac{d}{2^{J-2}}, \tilde Y^1, \Upsilon_{J-1}\bigg)\\
    &\leq \mathbb{P}\bigg(\|\mu - \Upsilon_J\| >  (\tilde C+1)(\frac{d}{2^{J}(\tilde C + 1)} + \rho) \bigg | \|\bar \nu - \Upsilon_{J-1}\|\leq \frac{d}{2^{J-2}}, \tilde Y^1, \Upsilon_{J-1}\bigg)\\
    & \leq  |M_{J-1}|\exp(-(\tilde C-2)^2(d/(2^J(\tilde C+1)) + \rho)^2/(16\sigma^2)))\\
    & \leq M^{\operatorname{loc}}\bigg(\frac{d}{2^{J - 2}}\bigg)\exp(-(\tilde C-2)^2(d/(2^J(\tilde C+1)) + \rho)^2/(16\sigma^2))).
\end{align*}
Since the bound doesn't depend on the value of $\Upsilon_{J-1}$, we can drop it from the conditioning.
Telescoping this bound by the union bound gives us that
\begin{align*}
    \mathbb{P}(\|\mu - \Upsilon_{J}\| > \rho + \frac{d}{2^{J-1}} | \tilde Y^1) &
    \leq M^{\operatorname{loc}}\bigg(\frac{d}{2^{J - 2}}\bigg) \sum_{j = 2}^{J} \exp(-(\tilde C-2)^2(d/(2^j(\tilde C+1)) + \rho)^2/(16\sigma^2)))\nonumber\\
    & \leq M^{\operatorname{loc}}\bigg(\frac{d}{2^{J - 2}}\bigg) \sum_{j = 2}^{J} \exp(-(\tilde C-2)^2(d/(2^{j}(\tilde C+1))^2/(16\sigma^2)))\\
    & \leq M^{\operatorname{loc}}\bigg(\frac{d}{2^{J - 2}}\bigg) a (1 + a^{4-1} + a^{16-1} + \ldots)\mathbbm{1}(J > 1)\nonumber\\
    & \leq M^{\operatorname{loc}}\bigg(\frac{d}{2^{J - 2}}\bigg) \frac{a}{1-a} \mathbbm{1}(J > 1)
\end{align*}
where for brevity we put $a =  \exp\bigg(\frac{-(\tilde C-2)^2d^2}{(2^{2 J}(\tilde C+1)^2)(16\sigma^2)}\bigg)$, and we are assuming that $ a < 1$. 

So if one sets $\varepsilon_J = \frac{(\tilde C-2)d}{2^{J}(\tilde C+1)}$, we have that if $ \varepsilon_J^2/(16\sigma^2) > 2 \log M^{\operatorname{loc}}\bigg(\varepsilon_J \frac{4 (\tilde C+1)}{(\tilde C-2)}\bigg)$ and \\$\exp(-\varepsilon_J^2/(16\sigma^2)) < 1/2$, the above probability will be bounded from above by $2\exp(-\varepsilon_J^2/(32\sigma^2))$. Since 
\begin{align*}
    2 \log M^{\operatorname{loc}}\bigg(\varepsilon_J \frac{4 (\tilde C+1)}{(\tilde C-2)}\bigg) \leq 2 \bigg(\log 2 \vee \log M^{\operatorname{loc}}\bigg(\varepsilon_J \frac{4 (\tilde C+1)}{(\tilde C-2)}\bigg)\bigg),
\end{align*} 
this condition is implied when $\frac{\varepsilon_J^2}{\sigma^2} > 32\log M^{\operatorname{loc}}\bigg(\varepsilon_J \frac{4 (\tilde C+1)}{(\tilde C-2)}\bigg) \vee 32\log 2$.

Below constants can change values from line to line. By the triangle inequality we have that $\|\nu^* - \mu\| \leq \|\nu^* - \Upsilon_J\| + \|\Upsilon_J - \mu\| \leq \rho + 6\varepsilon_J\frac{\tilde C+1}{\tilde C-2} \leq 7\varepsilon_J\frac{\tilde C+1}{\tilde C-2}$ with probability at least $1 - 2\exp(-\varepsilon_J^2/(32\sigma^2))$. 
Let $J^{**}$ be selected as the maximum $J$ such that $\frac{\varepsilon_J^2}{\sigma^2} > 32\log M^{\operatorname{loc}}\bigg(\varepsilon_J \frac{4 (\tilde C+1)}{(\tilde C-2)}\bigg) \vee 32\log 2$ otherwise if such $J$ does not exist $J^{**} = 1$.
We have shown that for all $J \leq J^{**}$ we have
\begin{align*}
    \MoveEqLeft \mathbb{P}(\|\mu - \nu^*\| > \frac{7}{2} \frac{d}{2^{J-1}}) \leq \underline{\underline C} \exp(-C' (d/2^{J-1})^2/\sigma^2)\mathbbm{1}(J^{**} > 1) \\
    &+ \mathbbm{1}\bigg(\frac{d}{2^{J + 1} (\tilde C + 1)} \leq 2\kappa \varepsilon^*\bigg) + C'' \exp(-C'''(d/2^{J-1})^2/\sigma^2)\mathbbm{1}(J^* > 1),
\end{align*}
where the last two summands, come from controlling the probability of the event $\frac{d}{2^{J + 1}(\tilde C + 1)} < \rho$. Hence for any $x \geq \varepsilon^{**} > 0$ (since if $\epsilon^{**} = 0$ then necessarily $\sigma = 0$ in which case the algorithm will return the point $\tilde Y^1 = \tilde Y^2 = \mu$) we have
\begin{align*}
    \MoveEqLeft \mathbb{P}(\|\mu - \nu^*\| \geq 8 x) \leq \mathbb{P}(\|\mu - \nu^*\| > 7x) \leq \underline{\underline C} \exp(-C' x^2/\sigma^2)\mathbbm{1}(J^{**} > 1) \\
    &+ \mathbbm{1}\bigg(\frac{x}{4 (\tilde C + 1)} \leq 2\kappa\varepsilon^*\bigg) + C'' \exp(-C'''x^2/\sigma^2)\mathbbm{1}(J^* > 1),
\end{align*}
where $\varepsilon^{**} = \varepsilon_{J^{**}}$. 



Integrating the tail bound as before we have
\begin{align*}
    \mathbb{E} \|\mu - \nu^*\|^2 
    & \leq C''' \varepsilon^{**2} + C^{''''} \sigma^2\exp(-C''\varepsilon^{**2}/\sigma^2)\mathbbm{1}(J^{**} > 1) \\
    & + C'''''\varepsilon^{*2} +C^{''''''} \sigma^2\exp(-\underline{\underline{\underline{C}}}\varepsilon^{*2}/\sigma^2)\mathbbm{1}(J^{*} > 1).
\end{align*}
Now $\varepsilon^{**2}/\sigma^2$ is bigger than a constant ($32 \log 2$) otherwise $J^{**} = 1$, and similarly for $\varepsilon^*$ and $J^*$. Hence the above is smaller than $\tilde C \max(\varepsilon^{*2}, \varepsilon^{**2})$ for some absolute constant $\tilde C$. Finally observe that $\varepsilon^{*}$ is smaller than $2\varepsilon^{***}$ which is defined as the infimum $\varepsilon$ such that
\begin{align*}
    \frac{\varepsilon^2}{\sigma^2} > 32\log M^{\operatorname{loc}}\bigg(\varepsilon \frac{2 (C+1)}{( C-2)}\bigg) \vee 32\log 2,
\end{align*}
since $M^{\operatorname{loc}}(x) \geq M_{B(\nu, 2^N)\cap K}^{\operatorname{loc}}(x)$ for any $x$. In addition, since $M^{\operatorname{loc}}\bigg(\varepsilon \frac{2 ( C+1)}{( C-2)}\bigg) \geq M^{\operatorname{loc}}\bigg(\varepsilon \frac{4 ( \tilde C+1)}{( \tilde C-2)}\bigg)$ (which follows since we have $\varepsilon \frac{2(\tilde C + 1)}{\tilde C -2} > \varepsilon \frac{C + 1}{ C - 2}$ and $c = 4(\tilde C+ 1) = 2(C + 1)$) we conclude that $2\varepsilon^{***} \geq \varepsilon^{**}$ . This completes the proof.
\end{proof}

\begin{remark}\label{not:closed:extension}
In this remark we explain how to fix the above proof for the case when the set $K$ is not necessarily closed. The issue lies in that in this case the estimators $\nu_m$ may not belong to the set $K$, and therefore we might not have a bound on the entropies localized at these points. The fix is simple. Since each estimator $\nu_m \in \overline K$ (where $\overline K$ is the closure of $K$), we can consider a sequence of points $\{\nu_{mi}\}_{i \in \NN}$ which has $\nu_m$ as its limiting point and each point $\nu_{mi} \in K$. For instance select each $\nu_{mi} = \alpha_i \nu + (1-\alpha_i)\nu_m$ for some appropriately chosen $\alpha_i$ which converges to $0$ (e.g. $\alpha_i = 1/i$). Note that this preserves measurability, and the selected $\nu_{mi}$ still belong to a compact set, yet are now points in the set $K$. 
Next instead of $\cE = \{\nu_m\}_{m \in \NN}$ in Algorithm \ref{test:v2} consider the countable set $\cE' = \{\nu_{mi}\}_{(m,i) \in \cI}$ where \begin{align*}
    \cI = \{(1,1), (1,2), (2,1), (3,1), (2,2), (1,3), (1,4), (2,3), (3,2), (4,1), (5,1), (4,2),\ldots,\}
\end{align*} (i.e. this sequence is usually used to prove that the rational numbers are countable). Note that inequality \eqref{to:modify:thm:unbounded} continues to hold since now $\nu_N$ is a limiting point of $\cE'$. Hence all arguments of the proof will remain valid. 

\end{remark}


\begin{proposition}\label{proposition:minimax:rate:unbounde:set}
Define $\varepsilon^*$ as $\sup \{\varepsilon: \varepsilon^2/\sigma^2 \leq \log M^{\operatorname{loc}}(\varepsilon)\}$, where $c$ in the definition of local entropy is a  sufficiently large absolute constant. Then the minimax rate is given by $\varepsilon^{*2}$ up to absolute constant factors.
\end{proposition}

\begin{proof}
For $\delta^* := \varepsilon^*/4$ we have $\log M^{\operatorname{loc}}(\delta^*) \geq\log M^{\operatorname{loc}}(\varepsilon^*) \geq  \varepsilon^{*2}/\sigma^2 = 16\delta^{*2}/\sigma^2$ and so this implies the sufficient condition for the lower bound (note that here we don't have a constant $4 \log 2$ per the comment in Section \ref{lower:boubd:section:no:constnat}). 

On the other hand we know that for a constant $C > 1$: 
\begin{align*}
    4 C \varepsilon^{*2}/\sigma^2 \geq C \log M^{\operatorname{loc}}(2\varepsilon^*) \geq C \log M^{\operatorname{loc}}(2\varepsilon^*\sqrt{C}) \geq C \log M^{\operatorname{loc}}\bigg(2\varepsilon^*\sqrt{C}\frac{c}{c/2-3}\bigg),  
\end{align*}
and so setting $\delta = 2 \varepsilon^* \sqrt{C}$ we obtain that 
\begin{align*}
    \delta^2/\sigma^2 \geq C \log M^{\operatorname{loc}}\bigg(\delta\frac{c}{c/2-3}\bigg).
\end{align*}
Plugging in $C = 32$ grants the requirement of Remark \ref{remark:no:constnant:upper:bound}, which completes the proof.

\end{proof}

\section{Discussion} \label{discussion:section}

In this paper we studied the minimax rate of the Gaussian sequence model under convex constraints. We proposed a method which is minimax optimal up to constant factors for any bounded convex set $K$, and an extension of the method which is minimax optimal for unbounded sets provided that $\sigma^2$ is known. Unfortunately, our algorithm is not computationally tractable. A natural open question is whether there exist computationally feasible general schemes which achieve the minimax rate for any set $K$. In addition, it is clear that the algorithm we proposed in this paper has something in common with the constrained LSE, as at each step it is looking for points which are closest to the observed point $Y$. It will be interesting if this connection is studied more closely --- in particular if there exist sufficient conditions for $K$ under which the two estimators are sufficiently close. Furthermore, throughout the paper we assumed that the model is well-specified, i.e., that $\mu \in K$. In future work we would like to see whether the techniques proposed here can capture the misspecified case. Another interesting open question is whether one can borrow ideas from this analysis to study the minimax risk under different loss functions, such as $\ell_p$ norms e.g. The biggest roadblock in terms of the upper bound that we currently see is extending Lemma \ref{important:lemma} to this more general setting. Finally an exciting question that remains is whether knowledge of $\sigma^2$ is necessary for the unbounded sets case. Our conjecture is that this is not the case, but at the moment we can only guarantee minimaxity by aggregating bounded estimators for which the knowledge of $\sigma^2$ seems to be required.

\section{Acknowledgements}

The author is grateful to Siva Balakrishnan for helpful discussions and for pointing him to the relevant papers by Li Zhang, to Ramon van Handel for enlightening discussions on entropy numbers, and to Larry Wasserman for encouragements. Thanks are also due to Shamindra Shrotriya who helped with plotting Figure \ref{fig:diagram}. Furthermore, the author would like to thank an AE and three anonymous referees for their insightful suggestions which greatly improved the presentation of this manuscript. The author was partially supported by grant NSF DMS-2113684.

\newpage

\appendix

\section{Finite Step Algorithm in the Presence of a Lower Bound of $\sigma$}\label{appendix:A}

The notation in this section is identical to the one used in Section \ref{upper:bound:section}.

\begin{algorithm}
\SetKwComment{Comment}{/* }{ */}
\caption{Upper Bound Algorithm with Finite Steps Given a Lower Bound on $\sigma$}\label{test:finite:version}
\KwInput{A point $\nu^* \in K$, $\overline J$ specified in Theorem \ref{upper:bound:rate:fininte}}
$k \gets 1$\;
$\Upsilon \gets [\nu^*]$ \Comment*[r]{This array is needed solely in the proof and is not used by the estimator}
\For{$k \leq \overline J$} {
    Take a $\frac{d}{2^{k}(C+1)}$ maximal packing set $M_k$ of the set $B\big(\nu^*, \frac{d}{2^{k-1}}\big) \cap K$ \Comment*[r]{The packing sets should be constructed prior to seeing the data}
    $\nu^* \gets \argmin_{\nu \in M_k} \|Y - \nu\| $ \Comment*[r]{Break ties by taking the point with the least lexicographic ordering}
    $\Upsilon$.append$(\nu^*)$\;
    $k \gets k + 1$\;
}
\Return{$\nu^*$}  
\end{algorithm}

\begin{theorem} \label{upper:bound:rate:fininte} Suppose $\underline \sigma$ is a known lower bound on $\sigma$. Let $\overline J$, be defined as the maximum integer $J$ such that
\begin{align}\label{upper:bound:suff:cond:fininte}
    \frac{\varepsilon_J^2}{\underline \sigma^2} > 16\log M^{\operatorname{loc}}\bigg(\varepsilon_J \frac{c}{(c/2-3)}\bigg) \vee 16\log 2,
\end{align}
where $\varepsilon_J := \frac{d (c/2 - 3)}{2^{J -2}c}$, and let $\overline J = 1$ if no such integer exists. Then estimator from Algorithm \ref{test:finite:version} returns a vector $\nu^*$ which satisfies the following property
\begin{align*}
    \mathbb{E} \|\mu - \nu^*\|^2 \leq \bar C \varepsilon^{*2},
\end{align*}
for some universal constant $\bar C$. Here $\varepsilon^*$ is the same as the one defined in equation \eqref{upper:bound:suff:cond} in Theorem \ref{upper:bound:rate}.
\end{theorem}
\begin{proof}
Combining the results of Lemma \ref{most:importnant:lemma} (with $c = 2(C+1)$ where $c$ is the constant from the definition of local packing entropy) and Lemma \ref{simple:lemma:monotone} we can conclude that 
\begin{align*}
    \mathbb{P}(\|\mu - \Upsilon_{J}\| > \frac{d}{2^{J-1}}) 
    &\leq M^{\operatorname{loc}}\bigg(\frac{d}{2^{J - 2}}\bigg) \sum_{j = 1}^{J-1}\exp\bigg(-\frac{(C-2)^2d^2}{(2^{2j}(C+1)^2)8\sigma^2}\bigg)\\
    & \leq M^{\operatorname{loc}}\bigg(\frac{d}{2^{J - 2}}\bigg) a (1 + a^{4-1} + a^{16-1} + \ldots)\mathbbm{1}(J > 1)\\
    & \leq M^{\operatorname{loc}}\bigg(\frac{d}{2^{J - 2}}\bigg) \frac{a}{1-a} \mathbbm{1}(J > 1),
\end{align*}
where for brevity we put \begin{align*}a =  \exp\bigg(\frac{-(C-2)^2d^2}{(2^{2(J - 1)}(C+1)^2)(8\sigma^2)}\bigg),
\end{align*}
and we are assuming that $ a < 1$. So if one sets $\varepsilon_J = \frac{(C-2)d}{2^{J - 1}(C+1)}$, we have that if $\varepsilon_J^2/(8\sigma^2) > 2 \log M^{\operatorname{loc}}\bigg(\varepsilon_J \frac{2 (C+1)}{(C-2)}\bigg)$ and $a = \exp(-\varepsilon_J^2/(8\sigma^2)) < 1/2$, the above probability will be bounded from above by $2\exp(-\varepsilon_J^2/(16\sigma^2))$. Since $2 \log M^{\operatorname{loc}}\bigg(\varepsilon_J \frac{2 (C+1)}{(C-2)}\bigg) < 2 \bigg(\log 2 \vee \log M^{\operatorname{loc}}\bigg(\varepsilon_J \frac{2 (C+1)}{(C-2)}\bigg)\bigg)$ this condition is implied when 
\begin{align}\label{suff:condition:epsJ:finite}
    \frac{\varepsilon_J^2}{\sigma^2} > 16\log M^{\operatorname{loc}}\bigg(\varepsilon_J \frac{2 (C+1)}{(C-2)}\bigg) \vee 16\log 2.
\end{align} 

By the triangle inequality we have that 
\begin{align}\label{mu:upislon:ineq:finite}
    \|\nu^* - \mu\| = \|\Upsilon_{\overline J} - \mu\| \leq \|\Upsilon_{\overline J} - \Upsilon_J\| + \|\Upsilon_J - \mu\| \leq 3\varepsilon_J\frac{C+1}{C-2},
\end{align} with probability at least $1 - 2\exp(-\varepsilon_J^2/(16\sigma^2))$ which holds for all $J$ satisfying \eqref{suff:condition:epsJ:finite} which include $\overline J$. Here we want to clarify that the last inequality in \eqref{mu:upislon:ineq:finite} follows from the fact that $\|\Upsilon_{\overline J} - \Upsilon_J\| \leq d/2^{J-2}$ when $\overline J \geq J$, as seen when we verified that $\Upsilon$ forms a Cauchy sequence.
Let $J^*$ be selected as the maximum $J$ such that \eqref{suff:condition:epsJ:finite} holds, or otherwise if such $J$ does not exist $J^* = 1$. Observe that the so defined $J^* \leq \overline J$, since $\underline \sigma \leq \sigma$ (which also holds in the case when $\overline J = 1$, because this implies $J^* = 1$). Let $\kappa = 3 \frac{C+1}{C-2}$, $\underline{C} = 2$ and $C' = \frac{1}{16}$. We have established that the following bound holds:
\begin{align*}
    \mathbb{P}(\|\mu - \nu^*\| > \kappa \varepsilon_J) \leq \underline C \exp(-C'\varepsilon_J^2/\sigma^2) \mathbbm{1}(J > 1) \leq \underline C \exp(-C'\varepsilon_J^2/\sigma^2) \mathbbm{1}(J^* > 1),
\end{align*}
for all $1 \leq J \leq J^*$, where this bound also holds in the case when $J^* = 1$ by exception. Observe that we can extend this bound to all $J \in \mathbb{Z}$ and $J \leq J^*$, since for $J < 1$ we have $\kappa \varepsilon_J \geq 6 d$ and so 
\begin{align*}
    \mathbb{P}(\|\mu - \nu^*\| > \kappa \varepsilon_J) \leq 0 \leq \underline C \exp(-C'\varepsilon_J^2/\sigma^2) \mathbbm{1}(J^* > 1).
\end{align*}
Now for any $\varepsilon_{J-1} > x \geq \varepsilon_{J}$ for $J \leq J^*$ we have that 
\begin{align*}
    \mathbb{P}(\|\mu - \nu^*\| > 2 \kappa x) & \leq \mathbb{P}(\|\mu - \nu^*\| \geq \kappa \varepsilon_{J-1}) \leq \underline C \exp(-C'\varepsilon_{J-1}^2/\sigma^2) \mathbbm{1}(J^* > 1)\\
    & \leq \underline C \exp(-C'x^2/\sigma^2)\mathbbm{1}(J^* > 1),
\end{align*}
where the last inequality follows due to the fact that the map $x \mapsto \underline C \exp(-C'x^2/\sigma^2)$ is monotonically decreasing for positive reals. We will now integrate the tail bound:
\begin{align*}
    \mathbb{P}(\|\mu - \nu^*\| \geq 3 \kappa x) \leq \mathbb{P}(\|\mu - \nu^*\| > 2 \kappa x) \leq \underline C \exp(-C'x^2/\sigma^2) \mathbbm{1}(J^* > 1),
\end{align*}
which holds true for $x \geq \varepsilon^*$ (for $\varepsilon^* > 0$; if $\varepsilon^* = 0$ we know $\sigma = 0$ and therefore $\underline{\sigma} = 0$ so we need to run the algorithm ad infinity (or simply output $Y$ in that case)), where $\varepsilon^* = \varepsilon_{J^*} = \frac{(C-2)d}{(C+1) 2^{J^* - 1}}$, always (since even if $J^* = 1$ by exception, this bound is still valid).

We have
\begin{align*}
    \mathbb{E} \|\mu - \nu^*\|^2 & = \int_{0}^{\infty} 2 x \mathbb{P}(\|\mu - \nu^*\| \geq x) dx \\
    & \leq C''' \varepsilon^{*2} + \int_{3\kappa\varepsilon^*}^{\infty} 2 x \underline C \exp(-C''x^2/\sigma^2) \mathbbm{1}(J^* > 1) dx \\
    & = C''' \varepsilon^{*2} + C^{''''} \sigma^2\exp(-C'''''\varepsilon^{*2}/\sigma^2)\mathbbm{1}(J^* > 1).
\end{align*}
Now $\varepsilon^{*2}/\sigma^2$ is bigger than a constant ($16 \log 2$) otherwise $J^* = 1$. Hence the above is smaller than $\bar C \varepsilon^{*2}$ for some absolute constant $\bar C$.
\end{proof}

\bibliographystyle{abbrvnat}
\bibliography{ref}

\end{document}